\documentclass[11pt]{amsart}
\usepackage{a4wide, amsmath, amssymb}
\usepackage[usenames]{color}
\newtheorem{thm}{Theorem}[section]
\newtheorem{cor}[thm]{Corollary}
\newtheorem{lem}[thm]{Lemma}

\theoremstyle{definition}
\theoremstyle{definition}
\theoremstyle{definition}
\newtheorem{ex}[thm]{Example}\theoremstyle{definition}

\newtheorem{rem}[thm]{Remark} \numberwithin{equation}{section}

\newcommand{\R}{\mathbb R}
\newcommand{\Z}{\mathbb Z}
\newcommand{\E}{\mathbb E}
\newcommand{\T}{\mathbb T}
\newcommand{\N}{\mathbb N}

\def\P{\mathbb P}
\def\1{\mathbb I}

\def\e{\epsilon}
\def\ep{\epsilon}

\begin{document}

\title[Gradient bounds and large time behavior]
{Gradient bounds for nonlinear degenerate parabolic equations
and application to large time behavior of systems}

\author{Olivier Ley\and Vinh Duc Nguyen}
\address{Olivier Ley (corresponding author): IRMAR, INSA de Rennes, 35708 Rennes, France} 
\email{olivier.ley@insa-rennes.fr}
\address{Vinh D. Nguyen: LAMA,  Universit\'e Paris-Est Cr\'eteil, 94010 Cr\'eteil, France} 
\email{vinh.nguyen@math.cnrs.fr}

\begin{abstract}
We obtain new oscillation and gradient bounds for the
viscosity solutions of fully nonlinear degenerate elliptic equations
where the Hamiltonian is a sum of a sublinear and a superlinear
part in the sense of Barles and Souganidis (2001).
We use these bounds to study the asymptotic behavior 
of weakly coupled systems of fully nonlinear parabolic
equations. Our results apply to some ``asymmetric systems''
where some equations contain a sublinear Hamiltonian
whereas the others contain a superlinear one.
Moreover, we can deal with some particular case
of systems containing some degenerate equations 
using a generalization of the strong maximum principle for systems.
\end{abstract}


\date{\today}

\maketitle


\section{Introduction}

One of the main result of this work is to obtain new results
about the large time behavior of the solution $u(x,t)=(u_1(x,t),\cdots, u_m(x,t))$
of the fully nonlinear parabolic system
\begin{eqnarray}\label{sys-general}
&& \left\{
\begin{array}{l}
\displaystyle\frac{\partial u_i}{\partial t}+\mathop{\rm sup}_{\theta\in\Theta}
\left\{-{\rm trace}(A_{\theta i}(x) D^2u_i)+H_{\theta i}(x,Du_i )\right\}+\sum_{j=1}^{m}d_{ij}u_j=0,\\
\hspace*{8cm}
(x,t)\in\T^N\times (0,+\infty), 
\\[3mm]
u_i(x,0)=u_{0i}(x),  \qquad x\in \T^N,
\end{array}
\right.
1\leq i\leq m,
\end{eqnarray}
in the periodic setting ($\T^N$ denotes the flat torus $\R^N/\Z^N$), where
the equations are linearly coupled through a matrix $D=(d_{ij})_{ij}$ which is assumed 
to be monotone and irreducible. The set $\Theta$ is a
metric space, the diffusion matrices can be written $A_{\theta i}(x)=\sigma_{\theta i}(x)\sigma_{\theta i}(x)^T$
with  $\sigma_{\theta i}$ bounded Lipschitz continuous and $H_i,u_{0i}$ are continuous.
\smallskip

To simplify the presentation, we present our results in the simplest case without dependence
with respect to $\theta$ and for $m=2.$ See Section~\ref{sec:exple-sys} 
for some discussions about the general case. We then consider
\begin{eqnarray}
\label{C}
&& \left\{
\begin{array}{l}
\displaystyle\frac{\partial u_1}{\partial t}-{\rm trace}(A_1(x) D^2u_1)+H_1(x,Du_1 )+u_1-u_2=0,\\[2mm]
\displaystyle\frac{\partial u_2}{\partial t}-{\rm trace}(A_2(x) D^2u_2)+H_2(x,Du_2 )+u_2-u_1=0,
\ \ \ (x,t)\in\T^N\times (0,+\infty),
\\[2mm]
u_1(x,0)=u_{01}(x), \ u_2(x,0)=u_{02}(x)  \qquad x\in \T^N,
\end{array}
\right.
\end{eqnarray}
The precise assumptions on the Hamiltonians $H_i$
will be explained below.
\smallskip

We prove the following asymptotic behavior of the solution,
\begin{eqnarray}\label{asympt-intro}
u_i(x,t)+c_i t\to v_i(x) \quad \text{uniformly in $\T^N$ as $t \to \infty$},
\text{ for $i=1,2$},
\end{eqnarray}
where $c=(c_1,c_2)\in\R^2$ and $v =(v_1,v_2)\in W^{1,\infty}(\T^N)^2$
are solutions of the so-called ergodic problem
\begin{eqnarray}
\label{E}
&& \left\{
\begin{array}{l}
\displaystyle-{\rm trace}(A_1(x) D^2v_1)+H_1(x,Dv_1 )+v_1-v_2=c_1,\\[2mm]
\displaystyle-{\rm trace}(A_2(x) D^2v_2)+H_2(x,Dv_2 )+v_2-v_1=c_2,\quad x\in\T^N.
\end{array}
\right.
\end{eqnarray}

Let us give immediately
one of the most striking application of our result.
We are able to prove the convergence for {\em asymmetric systems}
like
\begin{eqnarray}\label{exple-mixed}
&& \left\{
\begin{array}{ll}
\displaystyle\frac{\partial u_1}{\partial t}-{\rm trace}(A_1(x) D^2u_1)
+ \langle b_1(x),Du_1\rangle+\ell_1(x) +u_1 -u_2=0, & \\[3mm]
\displaystyle\frac{\partial u_2}{\partial t}-{\rm trace}(A_2(x) D^2u_2)
+|Du_2|^2+\ell_2(x) -u_1 + u_2=0, &
\end{array}
\right.
\end{eqnarray}
where $A_1$ is uniformly elliptic and $A_2$ may be degenerate.
The name asymmetric means that different equations can have different natures
as above: the first one contains a sublinear Hamiltonian and is uniformly
elliptic whereas the second one contains a superlinear one and may be
degenerate. The general framework is presented below. 
To explain the main difficulty to prove this result, let us recall
the related results for scalar equations.
\smallskip

To study the large time behavior for parabolic nonlinear equations
\begin{eqnarray}\label{cauchy-scalaire}
\displaystyle\frac{\partial u}{\partial t}
-{\rm trace}(A(x) D^2u) 
+H(x,Du )=0,\quad
(x,t)\in\T^N\times (0,+\infty), 
\end{eqnarray}
one has first to establish {\em uniform} gradient bounds
for the stationary equation
\begin{eqnarray}\label{approx-scal}
\e \phi^\e-{\rm trace}(A(x) D^2\phi^\e)
+  H(x, D\phi^\e)=0, & x\in\T^N, \quad \e >0.
\end{eqnarray}
By uniform gradient bounds, we mean
\begin{eqnarray}\label{grad-intro}
|D\phi^\e|_\infty\leq K, \ \ \ 
\text{where $K$ is independent of $\e.$}
\end{eqnarray}
This condition is crucial to be able to send $\e$ to 0 in~\eqref{approx-scal}
in order to solve the ergodic
problem~\eqref{E}, which is a first step when trying to
prove~\eqref{asympt-intro}.
\smallskip

Barles and Souganidis~\cite{bs01} obtained the first results concerning both
estimates like~\eqref{grad-intro} and the asymptotic behavior~\eqref{asympt-intro}
for scalar equations ($m=1$) with $A(x)=I$ in two contexts. The first one is
for {\em sublinear Hamiltonians}, i.e., for Hamiltonians with a sublinear growth 
with respect to the gradient. A typical example is
\begin{eqnarray*}
H(x,p)=\langle b(x),p\rangle+\ell(x), \quad \text{$b\in C(\T^N;\R^N),$ $\ell \in C(\T^N).$}
\end{eqnarray*}
The second context is
for {\em superlinear Hamiltonians}. The precise assumption is more involved
(see~\eqref{BS-originale}) and designed to allow the use of weak Bernstein-type
arguments (\cite{barles91a}). The most important example is Hamiltonians
with a superlinear growth with respect to the gradient
\begin{eqnarray}\label{typic-super-intro}
H(x,p)=a(x)|p|^{1+\alpha}+\ell(x), \quad \text{$\alpha >0,$ $a,\ell \in C(\T^N)$ and $a>0.$}
\end{eqnarray}
As a consequence of these bounds together with the strong maximum principle,
they obtain the convergence for the solutions of~\eqref{cauchy-scalaire}
when either $H$ is sublinear or $H$ is superlinear.
\smallskip

Using the extension of viscosity solutions to monotone systems
of parabolic equations by Ishii and Koike~\cite{ik91b}, it is
not difficult to adapt the results of~\cite{bs01} to the case
of {\em sublinear} systems~\eqref{C} (i.e., systems for which {\em all} the
Hamiltonians are sublinear in the sense
of~\cite{bs01}) on the one hand, and to the case of {\em superlinear}
systems on the other hand. 
In this work, we focus on the more delicate issue of {\em asymmetric
systems} like~\eqref{exple-mixed} containing both sublinear and superlinear
Hamiltonians. 
The main difficulty is that the proofs of~\eqref{grad-intro} 
in~\cite{bs01} are completely different in the two contexts.
In the case of sublinear Hamiltonians, a method due 
to~Ishii and Lions~\cite{il90} allows to take profit of
the uniform ellipticity of the equation
to control the sublinear terms coming from the Hamiltonian.
Whereas, up to our knowledge,
the strategy to prove~\eqref{grad-intro} for superlinear Hamiltonians
relies on a weak Bernstein
method needing an exponential-type change of function like 
$e^{w^\e}=\phi^\e - {\rm min}_{\T^N}\phi^\e+1$
in order to take advantage of the superlinear property of the Hamiltonian.
In the case of systems, one has to perform this
exponential change in all equations, producing 
quadratic terms of the form $|\sigma(x) Dw^\e|^2.$
These latter terms are dramatic in equations with sublinear
Hamiltonians since they are not anymore under the control
of the ellipticity using the usual proof.
\smallskip

The new idea to overcome this difficulty is
to establish first uniform oscillation bounds
\begin{eqnarray}\label{osc-intro}
{\rm osc}(\phi^\e):= \mathop{\rm sup}_{\T^N} \phi^\e
- \mathop{\rm inf}_{\T^N} \phi^\e \leq K \ \ \ 
\text{where $K$ is independent of $\e,$}
\end{eqnarray}
for the solution of~\eqref{approx-scal}.
The point is that the proof of the oscillation bound
does not need any exponential change of variable and so 
will work for asymmetric systems.
This gives some uniform bounds for the new function $w^\e$
and we are able to ``localize'' the proof (see~\eqref{xmoinsy} for details)
allowing to control the bad quadratic term.
It is worth mentioning that our proof of the oscillation bound works 
in very general settings and is a new result interesting by itself.
For instance,~\eqref{osc-intro} holds for the solutions of~\eqref{approx-scal} 
as soon as
\begin{eqnarray}\label{Hsur-lin}
\frac{H(x,p)}{|p|}\to +\infty \ \text{ as $|p|\to +\infty$ uniformly with respect to
$x.$}
\end{eqnarray}

Taking advantage of this oscillation bound, we are in fact able
to produce a kind of unified proof of the sublinear and superlinear
cases of~\cite{bs01}. More precisely, we obtain the gradient
bound~\eqref{grad-intro} for~\eqref{approx-scal} when 
\begin{eqnarray*}
H=\underline{H} + \overline{H},
\end{eqnarray*}
where $\underline{H}$ is a sublinear Hamiltonian, i.e, having a
sublinear growth,
\begin{eqnarray*}
|\overline{H}(x,p)|\leq C(1+|p|),
\end{eqnarray*}
and $\overline{H}$ is a superlinear one. 
An important feature, which will be crucial when dealing
with asymmetric systems, is that we allow 
$\underline{H}$ or/and $\overline{H}$ to be zero
which is natural for $\underline{H}$ but not so
for $\overline{H}.$ As far as the precise definition
of superlinear Hamiltonian is concerned, we propose two 
definitions,
one when~\eqref{approx-scal} is uniformly elliptic,
see~\eqref{BS-ell}, 
which generalizes slightly the one of~\cite{bs01}
and a stronger one, see~\eqref{BS}, which allows to deal with degenerate 
equations~\eqref{approx-scal}. Both of them include of course typical 
superlinear Hamiltonians like~\eqref{typic-super-intro} 
and  $\overline{H}$ may be zero in some cases.
We refer the reader to Section~\ref{sec:exple1} for comments and examples.
\smallskip

Using this new gradient bound for scalar equations, we are 
in position to extend it to systems
\begin{eqnarray}\label{syseps22intro}
&& \left\{
\begin{array}{l}
\e \phi_1^\e-{\rm trace}(A_1(x) D^2\phi_1^\e)+H_1(x,D\phi_1^\e )+\phi_1^\e-\phi_2^\e=0,\\[2mm]
\e \phi_2^\e-{\rm trace}(A_2(x) D^2\phi_2^\e)+H_2(x,D\phi_2^\e )+\phi_2^\e-\phi_1^\e=0, \ \ \ x\in\T^N.
\end{array}
\right.
\end{eqnarray}
An immediate consequence is that we can solve the ergodic problem~\eqref{E}
extending the classical by now proofs of \cite{lpv86, al98} to the case of our
systems. We then prove the convergence~\eqref{asympt-intro}.
As in~\cite{bs01}, the proof of the convergence is based on the strong maximum
principle but let us mention that we establish a new version of the strong
maximum principle for systems which may contain some degenerate equations,
see Theorem~\ref{smp-sys} for details. 
In particular, the result holds for~\eqref{exple-mixed};
See Section~\ref{sec:exple-sys} for more examples and discussions.
\smallskip

Let us turn to an overview of related results in the litterature.
The ideas of the proof of gradient bounds in viscosity theory
using the uniform ellipticity of the equation are due to Ishii and Lions~\cite{il90},
see also~\cite{cil92, barles08} and the references therein.
Gradient bounds for superlinear-type Hamiltonians can be found in
Lions~\cite{lions82} and Barles~\cite{barles91a}, see also Lions and Souganidis~\cite{ls05}.
These ideas were used in Barles and Souganidis~\cite{bs01}
as explained above. Our approach is mainly based on this latter work.
For superlinear Hamiltonians satisfying $H(x,p)\geq a^{-1}|p|^m-a,$ $a,m>1,$
some H\"older or gradient estimates
were obtained in Capuzzo Dolcetta et al.~\cite{clp10}, Barles~\cite{barles10},
Cardaliaguet~\cite{cardaliaguet09}. Recently, oscillations and H\"older bounds for nonlinear 
degenerate parabolic equations were proved in Cardaliaguet and Sylvestre~\cite{cs12}
but the bounds depends on the $L^\infty$ norm of the solution.
\smallskip

The large time behavior of such kind of nonlinear equations or systems 
in the periodic setting were
extensively studied. For Hamilton-Jacobi equations (the totally degenerate
case when $A\equiv 0$), we refer the reader 
to~\cite{nr99, fathi98, bs00, ds06, bim13}
and the references therein. In this framework, the gradient bounds
are not a difficult step but the proof of the convergence is more delicate
since one does not have any strong maximum principle. Such kind of results were
extended to systems of Hamilton-Jacobi equations in~\cite{clln12, mt12, nguyen12, mt13}.
For second order nonlinear equations, the asymptotics results of~\cite{bs01} were recently
generalized in~\cite{ln13} to some superlinear degenerate equations which are totally degenerate 
on some subset $\Sigma$ of $\T^N$ and uniformly parabolic outside $\Sigma$
 using the gradient
bound of Theorem~\ref{uni_grad_dege} and some strong maximum principle type ideas. 
Similar results for uniformly convex degenerate
equations were established in Cagnetti et al.~\cite{cgmt13} 
using a different
approach based on a nonlinear adjoint method~\cite{evans10}. 
Their results also apply to systems with uniformly convex quadratic Hamiltonians
with quite general degeneracy assumptions since
the proof is not based on strong maximum principle-type
arguments. However, such a method does not seem to be applicable for fully
nonlinear equations and the system in~\cite{cgmt13} is not asymmetric.
\smallskip

The paper is organized as follows.
The bounds~\eqref{osc-intro}-\eqref{grad-intro} are established in
Section~\ref{sec:grad} 
for the scalar equation~\eqref{approx-scal}
with $H=\underline{H} + \overline{H},$
firstly when the equation is uniformly elliptic and secondly in some
degenerate cases.
Some examples of applications are collected in Section~\ref{sec:exple1}.
Section~\ref{assym-sys} is devoted to systems. The gradient bounds
for asymmetric systems are obtained in Section~\ref{sec:grad-sys}
and a new strong maximum principle is obtained in Section~\ref{sec:smp}.
Then the ergodic problem is solved and the main application
of large time behavior of asymmetric is investigated.
The section ends with some examples of applications and extensions.
Several results are collected in the appendix.
In particular, since the equations under consideration do
not satisfy the classical assumptions in viscosity solutions
(due to the possibly superlinear growth of the Hamiltonian
for instance), we recall several versions of the comparison
principle which apply in our case. Finally a control theoritical interpretation
is given.
\smallskip

\noindent
{\bf Acknowledgement.} We thank Guy Barles, Hiroyoshi Mitake and
Hung Tran for fruitful discussions.
This work was partially supported by the ANR (Agence Nationale de
la Recherche) through HJnet project ANR-12-BS01-0008-01
and WKBHJ project ANR-12-BS01-0020.


\section{Gradient bounds for nonlinear parabolic equations}
\label{sec:grad}

For $\e >0,$ we consider the approximated equation
\begin{eqnarray}\label{AE-H}
&& \e \phi^\e  -{\rm trace}(A(x) D^2\phi^\e)
+H(x, D\phi^\e )=0, \quad x\in\T^N.
\end{eqnarray}
We will always assume
\begin{eqnarray}\label{cond-sig}
&& A = \sigma\sigma^T, \  \sigma\in W^{1,\infty}(\T^N;\R^{N\times N}),
\quad H \in C(\T^N\times \R^N).
\end{eqnarray}

\subsection{A general result for the oscillation}

We first show that the oscillation of the solution of~\eqref{AE-H}
is uniformly bounded under a very general hypothesis.
This result is interesting by itself.
\begin{eqnarray}\label{ssa4}
&&  \left\{\begin{array}{l}
\text{There exists $L>1$ such that
for all $x,y\in\T^N,$}\\
\text{if $|p|\!= \!L$, then }
\displaystyle  H(x,p) \ge|p|\left[H(y,\frac{p}{|p|})\!+\!|H(\cdot ,0)|_\infty
\!+\! N^{3/2} |\sigma_x|_\infty^2\right].
\end{array}\right.
\end{eqnarray}

Notice that~\eqref{ssa4} is satisfied when~\eqref{Hsur-lin} holds,
see Section~\ref{sec:exple1}.

\begin{lem}\label{oscillation}
Assume~\eqref{cond-sig} and~\eqref{ssa4}. Let  $\phi^\e$ be a continuous
solution of~\eqref{AE-H} and let $\phi^\e(x_\e)=\min \phi^\e$. Then
\begin{eqnarray*}
\phi^\e (x)- \phi^\e(x_\e)\le L|x-x_\e| \quad \text{for all $x\in \T^N$,}
\end{eqnarray*}
where $L$ is the constant (independent of $\e$) which appears in~\eqref{ssa4}.
\end{lem}

\noindent
An immediate consequence is
\begin{eqnarray*}
{\rm osc}(\phi^\e):=  \max \phi^\e-\min \phi^\e\leq \sqrt{N}L.
\end{eqnarray*}

\begin{proof}[Proof of Lemma \ref{oscillation}]
For simplicity, we skip the $\e$ superscript for $\phi^\e$ writing $\phi$
instead. The constant $L$ which appears below is the one of~\eqref{ssa4}. 
Consider
\begin{eqnarray*}
M=\max_{x,y \in \T^N}\{ \phi(x)-L \phi(y)+(L-1)\min\phi -L|x-y|\}.
\end{eqnarray*}
We are done if $M\le 0$.
Otherwise, the above positive maximum  is achieved at 
$(\overline{x},\overline{y})$ with $\overline{x}\not= \overline{y}$.
Notice that the continuity of $\phi$ is crucial at this step. 
The theory of second order viscosity solutions (see \cite{cil92} and Lemma~\ref{diff-tracet}) 
yields, for every $\varrho>0,$ the existence of 
$(p,X) \in \overline{J}^{2,+}\phi(\overline{x})$ and
$(p/L,Y/L) \in \overline{J}^{2,-}\phi(\overline{y})$, $p=L \frac{\overline{x}-\overline{y}}{|\overline{x}-\overline{y}|}$,
such that
\begin{eqnarray*}
&&-{\rm trace}(A(\overline{x})X
-A(\overline{y})Y)\geq 
- LN^{3/2} |\sigma_x|_\infty^2+O(\varrho)
\end{eqnarray*}
and 
\begin{eqnarray*}
&& \left\{
\begin{array}{ll}
\displaystyle \e \phi(\overline{x})-{\rm trace}(A(\overline{x}) X)+H(\overline{x},p)
\le 0,\\[2mm]
\displaystyle  \e \phi(\overline{y})-{\rm trace}(A(\overline{y})\frac{Y}{L})+H(\overline{y},\frac{p}{L})\} \ge 0.
\end{array}
\right.
\end{eqnarray*}
It follows
\begin{eqnarray*}
  \hspace*{0.2cm}  \e (\phi(\overline{x})-L\phi(\overline{y}))
-{\rm trace}(A(\overline{x}) X
-A(\overline{y})Y)
+H(\overline{x},p)
-L H(\overline{y},\frac{p}{L})
\leq 0.\nonumber
\end{eqnarray*}
We have 
\begin{eqnarray*}
\e (\phi(\overline{x})- L\phi(\overline{y}))>-(L-1)\e\min \phi
\ge -L|H(\cdot ,0)|_\infty
\end{eqnarray*}
since $\e\min \phi\leq |H(\cdot ,0)|_\infty$ by the maximum principle 
(see~\eqref{eps-phi-borne}).
Combining all the above information, we  get 
\begin{eqnarray*}
H(\overline{x},p)
-L \left[H(\overline{y},\frac{p}{L}) +|H(\cdot ,0)|_\infty+N^{3/2} |\sigma_x|_\infty^2\right]<0.
\end{eqnarray*}
Applying \eqref{ssa4} yields a contradiction.
\end{proof}

\subsection{Gradient bounds for uniformly elliptic equations}
\label{sec:bound-scalar}

In this section, we suppose that~\eqref{AE-H} is uniformly elliptic, i.e.,
\begin{eqnarray}\label{sig-deg}
&& \begin{array}{l}
\text{there exists } \nu >0 \text{ such that }
A(x)\geq \nu I, \quad
x\in\T^N.
\end{array}
\end{eqnarray}
In this setting, we consider Hamiltonians under the special
form $H=\underline{H} +\overline{H}$ 
where  $\underline{H}$ is a sublinear Hamiltonian and
$\overline{H}$ is of superlinear type in the sense defined
below. This form will be useful
later to deal with asymmetric systems.
We rewrite~\eqref{AE-H} as
\begin{eqnarray}\label{AE}
&& \e \phi^\e -{\rm trace}(A(x) D^2\phi^\e)
+\underline{H}(x, D\phi^\e ) +\overline{H}(x, D\phi^\e )=0.
\end{eqnarray}

We say that the $\underline{H}$ is a {\em sublinear Hamiltonian} if
\begin{eqnarray}\label{H3er}
|\underline{H}(x,p)| \le \underline{C}(1+|p|),\ (x,p)\in \T^N \times \R^N.
\end{eqnarray}

\noindent
We consider the following {\em superlinear}-type assumptions
for $\overline{H}.$ The first one is needed to obtain
an oscillation bound for the solution and the second one is slightly 
stronger to get the gradient bound.
\begin{eqnarray}\label{oscill-ell}
&&  \left\{\begin{array}{l}
\text{There exists $\overline{C}>0,$ $L ,\mu>1$ such that}\\
\text{for all $x,y\in\T^N,$ $|p| \ge  L$,  } 
\displaystyle  \ \overline{H}(x,p)-\mu\overline{H}(y,\frac{p}{\mu }) 
\ge -\overline{C}|p|.
\end{array}\right.
\end{eqnarray}

\begin{eqnarray}\label{BS-ell}
&&  \left\{\begin{array}{l}
\text{There exists $\overline{C}>0,$
$L>1$ such that}\\[2mm]
\text{for all $x,y\in \T^N$, $|p|\geq L,$
and $\mu\geq 1\!+\!L|x-y|,$ }\\
\displaystyle 
\overline{H}(x,p)-\mu\overline{H}(y,\frac{p}{\mu }) 
\ge -\overline{C}|p|.
\end{array}\right.
\end{eqnarray}
From the inequality in~\eqref{H3er}, we see that
$\underline{H}$ has a sublinear growth in the classical sense.
But, let us point out that the {\it superlinear} $\overline{H}$ 
may be zero in~\eqref{oscill-ell} and~\eqref{BS-ell}. This fact will allow
to treat some cases of asymmetric systems. We chose to keep the
terminology {\it superlinear} since~\eqref{BS-ell} is a consequence of
the superlinear-type assumption 
\begin{eqnarray}\label{BS-ell-origin}
&&  \left\{\begin{array}{l}
\overline{H}\in W_{\rm loc}^{1,\infty}(\T^N\times \R^N) 
\text{, and there exists $L>1$ such that}\\[2mm]
\text{for a.e. $x\in \T^N$, $|p|\geq L,$ }
\displaystyle 
L\left[\overline{H}_{p}p-\overline{H}
\right] -|\overline{H}_{x}|   \geq 0
\end{array}\right.
\end{eqnarray}
introduced in~\cite{bs01}. Moreover,~\eqref{BS-ell} is satisfied
for the typical superlinear Hamiltonian
$\overline{H}(x,p)=a(x)|p|^{1+\alpha}+\ell (x),$ $\alpha >0,$ $a>0$
we have in mind. We refer the reader to Section~\ref{sec:exple1}
for more discussions and examples showing that our assumptions are
quite general.

We state the main result of this section.

\begin{thm}\label{uni_grad}
Assume \eqref{cond-sig}, \eqref{sig-deg}, \eqref{H3er}
and \eqref{BS-ell}.
For all $\e >0,$ there exists a unique continuous viscosity solution 
$\phi^\e\in C(\T^N)$ of \eqref{AE} and a constant $K>0$ independent of $\e$ such that
\begin{eqnarray*}
|D\phi^\e|_{\infty} \le K.
\end{eqnarray*}
\end{thm}

The proof relies on two important lemmas. The first one establishes a
uniform bound for the oscillation and the second one improves this
bound into a gradient bound. We first state and prove the lemmas
and then give the proof of the theorem.

\begin{lem}\label{weakuniform2}
Under the hypotheses of Theorem~\ref{uni_grad}, where~\eqref{BS-ell} could be replaced by the weaker 
condition~\eqref{oscill-ell}, there exists a constant $K>0$ (independent of $\e$) 
such that, if $\phi^\e$ is a continuous solution of~\eqref{AE}, then 
\begin{eqnarray*}
{\rm osc}(\phi^\e):=  \max \phi^\e-\min \phi^\e\leq K.
\end{eqnarray*}
\end{lem}

\begin{proof}[Proof of Lemma \ref{weakuniform2}] 
For simplicity, we skip the $\e$ superscript for $\phi^\e$ writing $\phi$
instead.

\noindent{\it 1.  Construction a concave test function.}
Consider the function
\begin{eqnarray}\label{def-phi7}
&&\Psi(s)=\frac{A_1}{A_2}(1-{\rm e}^{-A_2 s})
\quad \text{for $0\leq s\leq \sqrt{N}={\rm diameter}(\T^N),$}
\end{eqnarray}
where $A_1,A_2>0$ will be chosen later.
It is straightforward to see that $\Psi$ is a $C^\infty$ concave increasing
function satisfying $\Psi(0)=0$ and, for all $s\in [0,\sqrt{N}],$
\begin{eqnarray}\label{equa-diff}
\Psi''+A_2\Psi'=0,\quad A_1 e^{-A_2 \sqrt{N}} = \Psi'(\sqrt{N})\leq \Psi'(s)\leq \Psi'(0)=A_1 \label{borne-derivee}.
\end{eqnarray}

\noindent{\it 2. Viscosity inequalities.}
Consider
\begin{eqnarray}\label{maxstar}
M_\mu=\max_{x,y \in \T^N}\{ \phi(x)+(\mu-1)\min\phi-\mu \phi(y)-\Psi(|x-y|)\},
\end{eqnarray}
with $\mu$ given in~\eqref{oscill-ell}.
If $M_\mu\le 0$ then the lemma holds with $K=A_1/A_2.$
From now on, we argue by contradiction
assuming that the maximum is positive and achieved at 
$(\overline{x},\overline{y})$
with $\overline{x}\not= \overline{y}$. 
 
The theory of second order viscosity solutions yields, for every
$\varrho>0,$ the existence of 
$(p,X) \in \overline{J}^{2,+}\phi(\overline{x}),(p/\mu,Y/\mu) \in \overline{J}^{2,-}\phi(\overline{y})$
such that
\begin{eqnarray}\label{mat}
\left(
\begin{array}{ccc}
X & 0 \\
0 & -Y
\end{array}
\right)
\le A+\varrho A^2,
\end{eqnarray}
with
\begin{eqnarray}
\label{mat-bis}
q=\frac{\overline{x}-\overline{y}}{|\overline{x}-\overline{y}|},
\quad p=\Psi'(|\overline{x}-\overline{y}|) q,
\quad B=\frac{1}{|\overline{x}-\overline{y}|} (I-q \otimes q),\\
\label{mat-ter}
A=\Psi'(|\overline{x}-\overline{y}|)
\left(
\begin{array}{ccc}
B & -B \\
-B & B
\end{array}
\right)
+\Psi''(|\overline{x}-\overline{y}|)
\left(
\begin{array}{ccc}
q \otimes q & -q \otimes q \\
-q \otimes q & q \otimes q
\end{array}
\right)
\end{eqnarray}
and the following viscosity inequalities hold for the solution $\phi$ of~\eqref{AE},
\begin{eqnarray*}
&& \left\{
\begin{array}{ll}
\displaystyle \e \phi(\overline{x})
-{\rm trace}(A(\overline{x}) X)+\underline{H}(\overline{x},p)+\overline{H}(\overline{x},p)
\le 0,\\[2mm]
\displaystyle \e \phi(\overline{y})-{\rm trace}(A(\overline{y})\frac{Y}{\mu})+\underline{H}(\overline{y},\frac{p}{\mu})+\overline{H}(\overline{y},\frac{p}{\mu}) \ge 0.
\end{array}
\right.
\end{eqnarray*}
It follows
\begin{eqnarray}
\label{visco-ineq185}
&& \e \phi(\overline{x})-\e \mu\phi(\overline{y})
-{\rm trace}(A(\overline{x}) X
-A(\overline{y})Y)\\\nonumber
&&\hspace*{4.3cm}+\underline{H}(\overline{x},p)
-\mu\underline{H}(\overline{y},\frac{p}{\mu})+\overline{H}(\overline{x},p)
-\mu\overline{H}(\overline{y},\frac{p}{\mu})
\leq 0.\nonumber
\end{eqnarray}

\noindent{\it 3. Trace estimates.}
We have the following estimates which will be useful in the sequel, 
see for instance~\cite{il90, bs01, barles08}. A proof is given in the
Appendix. 
\begin{lem}\label{diff-tracet}
Under assumption~\eqref{cond-sig},
\begin{eqnarray*}
&&-{\rm trace}(A(\overline{x})X
-A(\overline{y})Y)\geq 
- N |\sigma_x|_\infty^2 |\overline{x}-\overline{y}|\Psi'(|\overline{x}-\overline{y}|)+O(\varrho).
\end{eqnarray*}
If, in addition, \eqref{sig-deg} holds, then
\begin{eqnarray}\label{ineq-tracet}
-{\rm trace}(A(\overline{x})X
-A(\overline{y})Y)
\geq
-4\nu\Psi''(|\overline{x}-\overline{y}|)-\tilde{C}\Psi'(|\overline{x}-\overline{y}|)
+O(\varrho),
\end{eqnarray}
where $\tilde{C}=\tilde{C}(N,\nu,|\sigma|_\infty, |\sigma_x|_\infty)$
is given by~\eqref{def-ctilde}.
\end{lem}

\noindent{\it 4. End of the proof.} 
At $\hat{x}$ such that $\phi(\hat{x})={\rm max}\,\phi,$
we get
\begin{eqnarray*}
&& \e\min \phi + \underline{H}(\hat{x},0)+\overline{H}(\hat{x},0)
\leq \e\max \phi + 
\underline{H}(\hat{x},0)+\overline{H}(\hat{x},0)
\leq 0.
\end{eqnarray*}
It follows
\begin{eqnarray}\label{eps-min-phi}
\e\min \phi\leq |(\underline{H}+\overline{H}) 
(\cdot ,0) |_\infty =:\mathcal{H}_0.
\end{eqnarray}
So, using that the maximum is positive in~\eqref{maxstar}, we obtain
\begin{eqnarray*}
\e (\phi(\overline{x})- \mu\phi(\overline{y}))>-(\mu-1)\e\min \phi
\ge -(\mu -1)\mathcal{H}_0.
\end{eqnarray*}
We have
\begin{eqnarray*}
|\underline{H}(\overline{x},p)
-\mu\underline{H}(\overline{y},\frac{p}{\mu})|
\leq 2\underline{C}(1+|p|)= 2\underline{C}(1+\Psi'(|\overline{x}-\overline{y}|)),
\end{eqnarray*}
from~\eqref{H3er} and
\begin{eqnarray*}
&& -{\rm tr}(A(\overline{x})X
-A(\overline{y})Y) \geq -4\nu \Psi''(|\overline{x}-\overline{y}|)
-\tilde{C} \Psi'(|\overline{x}-\overline{y}|) + O(\varrho)
\end{eqnarray*}
from Lemma~\ref{diff-tracet} \eqref{ineq-tracet}.
Choosing 
\begin{eqnarray*}
A_1 = L e^{A_2\sqrt{N}},
\end{eqnarray*}
where $L$ is the constant of~\eqref{oscill-ell},
we obtain 
$|p|=\Psi'(|\overline{x}-\overline{y}|)\geq L$ from~\eqref{borne-derivee}, \eqref{mat-bis} 
and
\begin{eqnarray*}
 \overline{H}(\overline{x},p)-\mu\overline{H}(\overline{y},\frac{p}{\mu }) 
\ge -\overline{C}\Psi'(|\overline{x}-\overline{y}|)
\end{eqnarray*}
from~\eqref{oscill-ell}. 
Using these estimates in~\eqref{visco-ineq185}  and sending $\varrho$ to 0,
we have
\begin{eqnarray}\label{form-fin-123}
-4\nu  \Psi'' -(\tilde{C}+2\underline{C}+\overline{C}) \Psi'-(\mu -1)\mathcal{H}_0 -2\underline{C} \leq 0.
\end{eqnarray}
Recalling that  $\Psi''+A_2 \Psi'=0$ by~\eqref{borne-derivee}, we
obtain a contradiction with~\eqref{form-fin-123} with the choice
$$
A_2=\frac{1}{4\nu}\left( \tilde{C}+2\underline{C}+\overline{C}
+\frac{(\mu-1)\mathcal{H}_0 +2\underline{C}+1}{L}\right).
$$
It ends the proof.
\end{proof}

\begin{lem}\label{uniform2}
Under the hypotheses of Theorem~\ref{uni_grad}.
Let $\phi^\e$ be a continuous viscosity solution of \eqref{AE}
and define
\begin{eqnarray*}
\exp(w^\e)=\phi^\e-\min_{\T^N} \phi^\e+1.
\end{eqnarray*}
Then, there exists a constant $K$ (independent of $\e$)
such that $|Dw^\e|_{\infty}\le K.$
\end{lem}

\begin{proof}[Proof of Lemma \ref{uniform2}]
For simplicity, we skip the $\e$ superscript.

\noindent {\it 1. New equation for $w$.}
The function $w$ solves
\begin{eqnarray}\label{new-eq}
&& \e \,e^{-w(x)}(\min \phi -1)+\e 
-{\rm tr}(A D^2w)
+\underline{G}(x,w,Dw)+ \overline{G}(x,w,Dw)=0,
\end{eqnarray}
where
\begin{eqnarray}\label{new-hamilt}
\underline{G}(x,w,p)=e^{-w}\underline{H}(x,e^w p )-|\sigma(x)^Tp|^2,
\quad \overline{G}(x,w,p)=e^{-w}\overline{H}(x,e^w p ).
\end{eqnarray}

\noindent {\it 2. Definition of the test function.}
We define $\Psi(s)=\frac{A_1}{A_2}(1-{\rm e}^{-A_2 s})$ 
as in~\eqref{def-phi7} for $0\leq s \leq \sqrt{N}={\rm diameter}(\T^N)$
and set for further purpose
\begin{eqnarray}\label{choix-cstes1}
&& A_2=\frac{1}{4\nu}\left(\tilde{C}+4\underline{C}+\overline{C}
+\mathcal{H}_0+2|\sigma|_\infty|\sigma_x|_\infty {\rm osc}(\phi)\right)
\quad \text{and} \quad
A_1 =(L+{\rm osc}(\phi))
e^{A_2\sqrt{N}},
\end{eqnarray}
where $\nu, \underline{C}, \overline{C}, L$ are the constants appearing 
in the assumptions~\eqref{sig-deg},~\eqref{H3er},~\eqref{BS-ell}, $\mathcal{H}_0$ 
is defined in~\eqref{eps-min-phi}
and ${\rm osc}(\phi)$ is bounded independently of $\e$ by Lemma~\ref{weakuniform2}.

Since $\Psi(0)=0$ and
\begin{eqnarray*}
\Psi(\sqrt{N})= (L+{\rm osc}(\phi))\frac{e^{A_2\sqrt{N}}-1}{A_2}
\geq L+{\rm osc}(\phi)
> {\rm osc}(\phi),
\end{eqnarray*}
there exists $r\in [0,\sqrt{N}]$ such that
\begin{eqnarray}\label{choix-r}
&&\Psi(r)= {\rm osc}(\phi).
\end{eqnarray}

We then consider 
\begin{eqnarray}\label{max034}
\max_{x,y \in\T^N}\{ w(x)-w(y)-\Psi(|x-y|)\}.
\end{eqnarray}

If this maximum is nonpositive, then for all $x,y \in\T^N$, 
we have
\begin{eqnarray*}
w(x)-w(y)\le \Psi(|x-y|) \le A_1|x-y|,
\end{eqnarray*}
where the latter inequality follows from the concavity of $\Psi$. 
This yields the desired result. 

From now on, we argue by contradiction, assuming that
the maximum in~\eqref{max034} is positive.
This implies that it is achieved at 
$(\overline{x},\overline{y})\in\T^N\times\T^N$
with $\overline{x}\not= \overline{y}$ since 
$w$ is continuous.
Noticing that
\begin{eqnarray}\label{ineg135}
w(\overline{x})-w(\overline{y})
\leq e^{w(\overline{x})-w(\overline{y})}-1
\leq  e^{w(\overline{x})}-1
\leq \phi (\overline{x})- \mathop{\rm min}\,\phi 
\leq {\rm osc}(\phi),
\end{eqnarray}
we get
\begin{eqnarray*}
0 <w(\overline{x})-w(\overline{y})-\Psi (|\overline{x}-\overline{y}|)
\leq {\rm osc}(\phi)-\Psi (|\overline{x}-\overline{y}|).
\end{eqnarray*}
Using that $\Psi$ is increasing and~\eqref{choix-r},
we infer
\begin{eqnarray}\label{xmoinsy}
|\overline{x}-\overline{y}|<r.
\end{eqnarray}
This latter inequality is a kind of localization of points of maxima
in~\eqref{max034}. It will allow us to control the quadratic term coming
from~\eqref{new-hamilt} in term of the oscillation, see~\eqref{controleparosc}.

\noindent {\it 3. Viscosity inequalities for~\eqref{new-eq}.}
Writing the viscosity inequalities as in Step 2 of the proof
of Lemma~\ref{weakuniform2}, we obtain
\begin{eqnarray*}
&& \e e^{-w(\overline{x})}(\min \phi -1)+\e
 -{\rm trace}(A(\overline{x}) X)
+\underline{G}(\overline{x},w(\overline{x}),p)
+ \overline{G}(\overline{x},w(\overline{x}),p)\leq 0,\\
&& \e e^{-w(\overline{y})}(\min \phi -1)+\e
 -{\rm trace}(A(\overline{y}) Y)
+\underline{G}(\overline{y},w(\overline{y}),p)
+ \overline{G}(\overline{y},w(\overline{y}),p)\geq 0.
\end{eqnarray*}
Therefore,
\begin{eqnarray}\label{ineq3}
&&   \e (e^{-w(\overline{x})}-e^{-w(\overline{y})})(\min \phi -1)
-{\rm trace}(A(\overline{x}) X
-A(\overline{y})Y)\\\nonumber
&& \hspace*{1cm}+\underline{G}(\overline{x},w(\overline{x}),p)-\underline{G}(\overline{y},w(\overline{y}),p)
+\overline{G}(\overline{x},w(\overline{x}),p)- \overline{G}(\overline{y},w(\overline{y}),p)
\leq 0.
\end{eqnarray}
The end of the proof consists in reaching a contradiction in the above
inequality.

\noindent {\it 4. Estimates of the terms in~\eqref{ineq3}.}
From~\eqref{eps-min-phi} and the fact that $w(\overline{x})>w(\overline{y})\geq 0,$
we have
\begin{eqnarray}\label{ineg-stricte}
&& \e (e^{-w(\overline{x})}-e^{-w(\overline{y})})(\min \phi -1)
> -\mathcal{H}_0.
\end{eqnarray}

From Lemma~\ref{diff-tracet} \eqref{ineq-tracet}, we have
\begin{eqnarray}\label{ineq-trace840}
&& -{\rm trace}(A(\overline{x})X
-A(\overline{y})Y) \geq -4\nu \Psi''(|\overline{x}-\overline{y}|)
-\tilde{C} \Psi'(|\overline{x}-\overline{y}|) + O(\varrho).
\end{eqnarray}

Using~\eqref{cond-sig}, \eqref{H3er} and recalling that
$|p|=\Psi'(|\overline{x}-\overline{y}|),$ we get
\begin{eqnarray}\label{estim-term-souslin}
&& |\underline{G}(\overline{x},w(\overline{x}),p)
-\underline{G}(\overline{y},w(\overline{y}),p)|\\\nonumber
&\leq &
|e^{-w(\overline{x})}\underline{H}(\overline{x},e^{w(\overline{x})} p )|
+ |e^{-w(\overline{y})}\underline{H}(\overline{y},e^{w(\overline{y})} p )|
+ ||\sigma (\overline{x})^Tp|^2-|\sigma(\overline{y})^Tp|^2|\\\nonumber
&\leq &
2\underline{C}(1+\Psi'(|\overline{x}-\overline{y}|))
+ 2|\sigma|_\infty |\sigma_x|_\infty
 |\overline{x}-\overline{y}| {\Psi'}^2(|\overline{x}-\overline{y}|).
\end{eqnarray}

We now estimate $\overline{G}(\overline{x},w(\overline{x}),p)
- \overline{G}(\overline{y},w(\overline{y}),p)$ using~\eqref{BS-ell}.
Set $P:= e^{w(\overline{x})}p$ and $\mu:= e^{w(\overline{x})-w(\overline{y})},$
we have
\begin{eqnarray}\label{diffG45}
\overline{G}(\overline{x},w(\overline{x}),p)
- \overline{G}(\overline{y},w(\overline{y}),p)
&=&
e^{-w(\overline{x})}\left(
\overline{H} (\overline{x}, P)
- \mu
\overline{H} (\overline{y}, \frac{P}{\mu})
\right).
\end{eqnarray}
From the choice of $A_1$ in~\eqref{choix-cstes1} and the concavity
of $\Psi,$ we get
\begin{eqnarray}\label{cond-ell-2}
&&|P|\geq |p|=\Psi'(|\overline{x}-\overline{y}|) \geq \Psi'(r)=A_1 e^{-A_2 r}
=(L+{\rm osc}(\phi)) e^{A_2(\sqrt{N}-r)}\geq L.
\end{eqnarray}
Since the maximum in~\eqref{max034} is positive, it follows
\begin{eqnarray}\label{cond-mu12}
&& \mu \geq 1+ w(\overline{x})-w(\overline{y}) > 1+ \Psi(|\overline{x}-\overline{y}|)
\geq 1+ \Psi'(|\overline{x}-\overline{y}|)|\overline{x}-\overline{y}|\geq 1+ L|\overline{x}-\overline{y}|.
\end{eqnarray}
From~\eqref{ineg135} and since $|p|\leq \Psi'(0)= A_1,$ we notice
\begin{eqnarray}\label{pborne}
&& L\leq |p|\leq A_1= A_1(\sigma, \underline{C}, \overline{C}, L, {\rm osc}(\phi))
\quad \text{and} \quad
1+ L |\overline{x}-\overline{y}|\leq \mu\leq  {\rm osc}(\phi)+1,
\end{eqnarray}
which will be useful in the proof of Theorem~\ref{uni_grad}.
It follows that we can apply~\eqref{BS-ell} to~\eqref{diffG45} to get
\begin{eqnarray}\label{estim-surlin}
&& \overline{G}(\overline{x},w(\overline{x}),p)
- \overline{G}(\overline{y},w(\overline{y}),p)
\geq -\overline{C}e^{-w(\overline{x})}|P|= -\overline{C}\Psi'(|\overline{x}-\overline{y}|).
\end{eqnarray}

Plugging  \eqref{ineg-stricte}, \eqref{ineq-trace840}, \eqref{estim-term-souslin}
and~\eqref{estim-surlin} in~\eqref{ineq3}  and by letting $\varrho\to 0,$ we obtain
\begin{eqnarray*}
&& -4\nu \Psi''(|\overline{x}-\overline{y}|)  
- (\tilde{C}+2\underline{C}+\overline{C})\Psi'(|\overline{x}-\overline{y}|)
-2|\sigma|_\infty |\sigma_x|_\infty |\overline{x}-\overline{y}| 
{\Psi'}^2(|\overline{x}-\overline{y}|) -\mathcal{H}_0-2\underline{C} <0.
\end{eqnarray*}
Since $|\overline{x}-\overline{y}|\leq r$ by~\eqref{xmoinsy}, using that
for all $s\in [0,r],$ $\Psi''(s)+A_2\Psi'(s)=0$ and
\begin{eqnarray}\label{controleparosc}
&&|\overline{x}-\overline{y}| {\Psi'}^2(|\overline{x}-\overline{y}| )
\leq \Psi(|\overline{x}-\overline{y}| )\Psi'(|\overline{x}-\overline{y}| )
\leq \Psi(r)\Psi'(|\overline{x}-\overline{y}| )= {\rm osc}(\phi)\Psi'(|\overline{x}-\overline{y}| ), 
\end{eqnarray}
we can rewrite the above estimate as
\begin{eqnarray*}
\left( 4\nu A_2 -(\tilde{C}+2\underline{C}+\overline{C}
+2|\sigma|_\infty |\sigma_x|_\infty{\rm osc}(\phi))\right)\Psi'(|\overline{x}-\overline{y}|)
-\mathcal{H}_0-2\underline{C} <0.
\end{eqnarray*}
It is then straightforward to see that~\eqref{choix-cstes1} leads
to a contradiction in the above inequality.
Finally, we obtain the result with $K=A_1.$
\end{proof}

We turn to the proof of Theorem~\ref{uni_grad}.
To use the previous lemmas, we need first to build a 
{\em continuous} viscosity solution of~\eqref{AE}.
It is straightforward to build a discontinuous viscosity solution
for~\eqref{AE} using Perron's method for viscosity solutions
of second order equations, see \cite{ishii87, cil92}.
The continuity of this solution usually follows from a  strong comparison
principle for~\eqref{AE}, i.e., a comparison
principle between
USC viscosity subsolutions and LSC viscosity supersolutions.
But, classical structure assumptions on the first order nonlinearity
$\underline{H}+\overline{H}$ like Lipschitz continuity
in $x$-variable, do not hold here. It follows that
a strong comparison result may not hold. We are only able to compare
a discontinuous sub- or supersolution with a Lipschitz continuous solution,
see Theorem~\ref{comp-avec-lip} 
in the Appendix. That is why, we need another approach inspired from~\cite{bs01} to
build a continuous solution of~\eqref{AE}. It is based on a truncation of 
the nonlinearity. Another natural approach would be to use  the
classical regularity theory for uniformly elliptic equations but it would not
apply for degenerate equations we will consider in the next section.

\begin{proof}[Proof of Theorem~\ref{uni_grad}]
We fix $\e>0$ and skip the $\e$-dependence in the proof for simplicity.

\noindent
{\it 1. Truncated Hamiltonian.}
For all $n>0,$ we define the continuous Hamiltonian
\begin{eqnarray}\label{H-approche}
&& H_{ n}:= \underline{H}_{ n}+\overline{H}_{ n} 
\end{eqnarray}
with
\begin{eqnarray*}
&& \underline{H}_{n}(x,p)\text{ (resp. $\overline{H}_{n}(x,p)$)}=\left\{ 
\begin{array}{ll}
\underline{H}(x,p)\text{ (resp. $\overline{H}(x,p)$)}
& \text{if } |p|\leq n,\\
\underline{H}(x,n\frac{p}{|p|})
\text{ (resp. $\overline{H}(x,n\frac{p}{|p|})$)} & \text{if } |p|\geq n.
\end{array}
\right.
\end{eqnarray*}

\noindent
{\it 2.  Construction of a continuous viscosity solution
for the truncated equation.}
We have a strong comparison principle between discontinuous
solutions for~\eqref{AE} where $\underline{H},$
$\overline{H}$ are replaced with $\underline{H}_{n}, 
\overline{H}_{n}$ respectively,
see Theorem~\ref{comp-discont}.
By Perron's method, see \cite{cil92, ishii87}, there exists
a continuous viscosity solution $\phi_n.$ 

\noindent
{\it 3. Uniform Lipschitz bound for solutions
of the truncated equation.}
We first notice that $\underline{H}_{n}$ satisfies~\eqref{H3er}
with the same constant $\underline{C}$ as $\underline{H}.$
Moreover, if we choose $n$ bigger than $L$ which appears in~\eqref{oscill-ell},
then, by Lemma~\ref{weakuniform2}, we obtain a bound for the oscillation
of $\phi_n$ which is independent of $n, \e.$ Moreover, if $n$
is chosen bigger than the right hand side of \eqref{pborne},
then~\eqref{BS-ell} hold for $\overline{H}_{n}$ with
the same constants as for $\overline{H}$ for all
$p$ satisfying \eqref{pborne}. As noticed in the proof of 
Lemma~\ref{uniform2}, it is enough to obtain 
a gradient bound $K$ for $w_n$ defined by
\begin{eqnarray*}
{\rm exp}(w_n)=\phi_n- \mathop{\rm min}_{\T^N}\phi_n +1.
\end{eqnarray*}
The crucial point is that this gradient bound $K$ does not depend on $n$
since the constants in~\eqref{H3er},~\eqref{oscill-ell},~\eqref{BS-ell}  
are the same for all $\underline{H}_{n}, \overline{H}_{n}.$
The uniform bound for the oscillation of $\phi_n$ yields
a $L^\infty$ bound for $w_n$
which is independent of $n, \e.$ It follows
\begin{eqnarray}\label{borne-grad-phin}
|D\phi_n|_\infty\leq K {\rm exp}(2\sqrt{N}K).
\end{eqnarray}

\noindent
{\it 4. Convergence of $\phi_n.$}
In addition to~\eqref{borne-grad-phin}, from~\eqref{eps-phi-borne},
we have a $L^\infty$ bound for $\phi_n$ which is independent of $n.$
Then, by Ascoli-Arzela's theorem, we obtain that, up to a subsequence, 
$\phi_n \to \phi$ in $C(\T^N)$
and $\phi=\phi^\e$ is Lipschitz continuous with $|D\phi|_\infty \leq K e^{2\sqrt{N}K},$
which is independent of $\e.$
Noticing that  $\underline{H}_{n}\to 
\underline{H}, \overline{H}_{n}\to 
\underline{H}$. By the stability
result for viscosity solution, we conclude that $\phi$ is a Lipschitz
continuous viscosity solution of~\eqref{AE}. 

\noindent
{\it 5. Uniqueness.}
The uniqueness of $\phi$ in the class of continuous viscosity solutions
relies on a comparison principle between the Lipschitz solution $\phi$
with any continuous solution $\tilde{\phi}$ of~\eqref{AE}, see 
Theorem~\ref{comp-avec-lip}
in the Appendix.
\end{proof}

\subsection{Gradient bounds for degenerate elliptic equations}
\label{sec:bound-scalar-dege}

We now consider degenerate elliptic equations, i.e.,~\eqref{sig-deg} does not
necessarily hold. In this case, we suppose
that the sublinear Hamiltonian $\underline{H}\equiv 0$ in \eqref{AE} and we 
reinforce the assumptions~\eqref{oscill-ell}-\eqref{BS-ell}
in order that the superlinear Hamiltonian $\overline{H}$
controls all the terms in the equation.

The assumptions are
\begin{eqnarray}\label{ssa5}
&&  \left\{\begin{array}{l}
\text{There exists $L,\mu>1$ 
such that}\\[2mm] 
\text{for all  $x,y\in\T^N,$ $|p|\ge L$,}\\[2mm]
\displaystyle  \overline{H}(x,p)-\mu \overline{H}(y,\frac{p}{\mu}) 
\ge (\mu-1)\left( |\overline{H}(\cdot ,0)|_\infty
+ N |\sigma_x|_\infty^2|p|\right),
\end{array}\right.
\end{eqnarray}
\begin{eqnarray}\label{BS}
&&  \left\{\begin{array}{l}
\text{There exists $L>1$ such that}\\[2mm]
\text{for all $x,y\in\T^N,$ $|p|\geq L$
and $\mu\geq 1+L|x-y|,$}\\[2mm]
\displaystyle 
\overline{H}(x,p)-\mu \overline{H}(y,\frac{p}{\mu}) 
\ge (\mu-1)\left( |\overline{H}(\cdot ,0)|_\infty
+ N |\sigma_x|_\infty^2|p|+ 2|\sigma_x|_\infty
|\sigma|_\infty|p|\right).
\end{array}\right.
\end{eqnarray}
As for the uniformly elliptic case, the first assumption is needed to get 
the oscillation bound and the stronger one to obtain the gradient bound.
Some discussion about these assumptions and examples
are given in Section~\ref{sec:exple1}.

\begin{thm}\label{uni_grad_dege}
Assume \eqref{cond-sig},~\eqref{BS}
and suppose that  $\underline{H}\equiv 0$.
For all $\e >0,$ there exists a unique continuous viscosity solution 
$\phi^\e\in C(\T^N)$ of \eqref{AE} and a constant $K>0$ independent of 
$\e$ such that
\begin{eqnarray*}
|D\phi^\e|_{\infty} \le K.
\end{eqnarray*}
\end{thm}


The proof of the above theorem is similar to the one of
Theorem~\ref{uni_grad}, so we skip it. It relies on the auxiliary
Lemmas \ref{weakuniform2} and~\ref{uniform2}
where
\eqref{oscill-ell}-\eqref{BS-ell} are replaced by \eqref{ssa5}-\eqref{BS}.
The proofs of the auxiliary lemmas follow the same lines, the changes
occur in the estimates of the terms in~\eqref{visco-ineq185}
and \eqref{ineq3}, so we only rewrite Step 4 of the proof of Lemma~\ref{uniform2}
where the main changes occur.

\begin{proof}[Proof of Lemma~\ref{uniform2} for Theorem~\ref{uni_grad_dege}
under assumptions~\eqref{BS}]
We estimate the different terms appearing in~\eqref{ineq3}.
We set $P:= e^{w(\overline{x})}p$ and $\mu:= e^{w(\overline{x})-w(\overline{y})}$
and we recall that, if the maximum is positive in~\eqref{max034}
and with a suitable choice of $A_1,A_2,r$ in~\eqref{def-phi7}, then
\begin{eqnarray*}
|P|=|e^{w(\overline{x})}p|\geq L>1
\text{ and } 
\mu > 1+ |p||\overline{x}-\overline{y}|\geq  1+ L|\overline{x}-\overline{y}|
\end{eqnarray*}
(see~\eqref{cond-ell-2} and~\eqref{cond-mu12}).
From~\eqref{eps-min-phi}, we get
\begin{eqnarray}\label{strict365}
 \e (e^{-w(\overline{x})}-e^{-w(\overline{y})})(\min \phi -1)
&>& e^{-w(\overline{x})}(1-\mu)\e(\min \phi -1)\\ \nonumber
&\geq& -e^{-w(\overline{x})}(\mu-1)|\overline{H}(\cdot,0)|_\infty.
\end{eqnarray}
From Lemma~\ref{diff-tracet} \eqref{ineq-tracet}, we have
\begin{eqnarray*}
&& -{\rm trace}(A(\overline{x})X
-A(\overline{y})Y) \geq 
-e^{-w(\overline{x})}(\mu -1) N|\sigma_x|_\infty^2 |P|
+ O(\varrho).
\end{eqnarray*}
Since  $\underline{H}\equiv 0,$ we have
\begin{eqnarray*}
\underline{G}(\overline{x},w(\overline{x}),p)-\underline{G}(\overline{y},w(\overline{y}),p)
&=& -|\sigma(\overline{x})^Tp|^2 + |\sigma(\overline{y})^Tp|^2\\
&\geq &
-2 |\sigma|_\infty |\sigma_x|_\infty |\overline{x}-\overline{y}| |p|^2 \nonumber\\
&\geq &
-2e^{-w(\overline{x})}(\mu -1)|\sigma|_\infty |\sigma_x|_\infty |P|.\nonumber
\end{eqnarray*}
As far as the superlinear Hamiltonians are concerned, we have
\begin{eqnarray*}
\overline{G}(\overline{x},w(\overline{x}),p)
- \overline{G}(\overline{y},w(\overline{y}),p)
&=&
e^{-w(\overline{x})}\left(
\overline{H} (\overline{x}, P)
- \mu
\overline{H} (\overline{y}, \frac{P}{\mu})
\right).
\end{eqnarray*}
Plugging the previous estimates in~\eqref{ineq3} and 
applying~\eqref{BS}, we reach a contradiction
with the strict inequality in~\eqref{strict365}.
It ends the proof.
\end{proof}


\subsection{Comments on the assumptions, examples and extensions}
\label{sec:exple1}



\begin{ex} (Hamiltonian satisfying~\eqref{ssa4}) 
If $\mathop{\rm lim\,sup}_{|p|\to +\infty} \frac{H(x,p)}{|p|}=+\infty$
uniformly with respect to $x$
then~\eqref{ssa4} holds. For instance, if
$H (x,p)=a(x) {\rm sin} (h(|p|)) |p|^{1+\alpha}+\ell(x),$
where $\alpha>0,$ $a>0,\ell,h$ are continuous and  $\mathop{\rm lim\,sup}_{r\to +\infty} {\rm sin}(h(r))>0,$
then~\eqref{ssa4} holds. 
\end{ex}

\begin{lem}\label{lem-osc2}\
\begin{itemize}
\item[(i)] A sublinear Hamiltonian $\underline{H}$ satisfying~\eqref{H3er}
satisfies~\eqref{oscill-ell}.

\item[(ii)] If there exists $\alpha \geq 0,$ $a>0,$ $A, B,L\geq 0$
such that, for all $x\in\T^N,$ $|p|\geq L,$ 
\begin{eqnarray*}
A|p|^\alpha + B|p|\geq H(x,p)\geq a|p|^\alpha -B|p|,
\end{eqnarray*}
then~\eqref{oscill-ell} holds 

\end{itemize}
\end{lem}

\begin{proof}[Proof of Lemma~\ref{lem-osc2}]
(i) Using \eqref{H3er}, we have, for $\mu=2$ and $|p|>1,$ 
\begin{eqnarray*}
&& \underline{H}(x,p)-\mu \underline{H}(y,\frac{p}{\mu})
\geq 
-C(1+|p|)-C\mu (1+|\frac{p}{\mu}|)
\geq 
-3C(1+|p|)
\geq -6C|p|.
\end{eqnarray*}
(ii) When $0\leq \alpha\leq 1,$ then \eqref{H3er} holds.
For $\alpha >1,$
 \begin{eqnarray*}
&& {H}(x,p)-\mu {H}(y,\frac{p}{\mu})
\geq 
(a-\mu^{1-\alpha}A)|p|^\alpha -2B|p|\geq -2B|p|
\end{eqnarray*}
provided $\mu \geq (A/a)^{1/(\alpha -1)}.$
\end{proof}

Notice that no regularity assumption (except continuity) is needed
to obtain the oscillation bound. To obtain a gradient bound, we need
to reinforce~\eqref{oscill-ell}-\eqref{ssa5} into~\eqref{BS-ell}-\eqref{BS}.
These new assumptions contain a kind of regularity assumption with
respect to $(x,p).$ Actually,~\eqref{BS-ell} is very close to~\eqref{BS-ell-origin} 
in~\cite{bs01} and~\eqref{BS} is very close to
\begin{eqnarray}\label{BS-originale}
&&  \left\{\begin{array}{l}
\overline{H}\in W_{\rm loc}^{1,\infty}(\T^N\times \R^N) 
\text{ and there exists $L$ such that}\\[2mm]
\text{if $|p|\geq L,$ then,  for a.e. $(x,p)\in \T^N \times \R^N,$}\\[2mm]
\displaystyle 
L\left[(\overline{H})_{p}p-\overline{H}
-|\overline{H}(\cdot,0)|_\infty
-N|\sigma_x|_\infty^2|p|
-2|\sigma|_\infty |\sigma_x|_\infty | p|
\right]
-|(\overline{H})_{x}| \geq 0,
\end{array}\right.
\end{eqnarray}
which is an extension of~\eqref{BS-ell-origin} for $x$-dependent
degenerate diffusion matrices.
In~\eqref{BS-ell-origin}-\eqref{BS-originale}, the Hamiltonian is
supposed to be locally Lipschitz with respect to $(x,p).$ When the
Hamiltonians are locally Lipschitz, we can prove that
\eqref{BS-ell}-\eqref{BS} and \eqref{BS-ell-origin}-\eqref{BS-originale}
are equivalent but our assumptions allow to deal with some
non Lipschitz continuous Hamiltonians as shown in the following
example. 

\begin{ex} (A non Lipschitz continuous Hamiltonian satisfying~\eqref{BS-ell}) 
Let $H(x,p)=|p|^2 + h(x,p)$ with $h$ continuous bounded.
For all $x,y\in\T^N, p\in\R^N$ and $\mu >1,$ we have
\begin{eqnarray*}
{H}(x,p)-\mu {H}(y,\frac{p}{\mu})=(1-\frac{1}{\mu})|p|^2
+ h(x,p)-\mu h(y,\frac{p}{\mu})
\geq -(1+\mu )|h|_\infty.
\end{eqnarray*}
From~\eqref{pborne}, we see that~\eqref{BS-ell} has to hold only for bounded
$\mu\geq 1+L|x-y|$ and $|p|\geq L>1.$ So~\eqref{BS-ell} holds. 
\end{ex}

We turn to some examples of sub- and superlinear Hamiltonians.

\begin{ex} (Typical sublinear Hamiltonians from optimal control problem)
\begin{eqnarray*}
\underline{H}(x,p)= \mathop{\rm sup}_{\theta\in\Theta}\{
- \langle b_\theta (x), p\rangle -\ell_\theta (x)\},
\end{eqnarray*}
where $b_\theta, \ell_\theta \in W^{1,\infty}(\T^N)$ uniformly with respect to $\theta.$
Such a $\underline{H}$ satisfies~\eqref{H3er}. Notice that the Lipschitz
continuity is actually not needed in the proof of the gradient bound.
\end{ex}

\begin{ex} (Typical superlinear Hamiltonian)
\begin{eqnarray*}
\overline{H}(x,p)= a(x)|p|^{1+\alpha}+b(x)|p|+c(x),
\end{eqnarray*}
with $a>0,b,c \in W^{1,\infty}(\T^N)$.
Then $\overline{H}$ satisfies~\eqref{BS-ell} and~\eqref{BS}
as soon as~\eqref{cond-sig} holds. 
It follows that the gradient bound of Theorem~\ref{uni_grad} holds 
for~\eqref{AE} even if $\sigma$ is degenerate
(i.e., \eqref{sig-deg} does not hold).
\end{ex}

\begin{ex}
$H(x,p)=|B(x)p|^k+\langle b(x),p\rangle +\ell(x)$
with $B\in C(\T^N ;\R^{N\times N}),$ $b\in C(\T^N;\R^N),$ $\ell\in C(\T^N)$
satisfies~\eqref{BS-ell} if
\begin{eqnarray*}
&& |b_x|_\infty, |B_x|_\infty\leq C \quad \text{and} \quad B(x)B(x)^T > 0. 
\end{eqnarray*}
\end{ex}

\begin{ex}
Let define $\hat{H}(p)=\hat{H}(|p|)$ radial by
$\hat{H}(0)=0$ and, for all $t\in [n,n+1],$ 
$\hat{H}(t)= (n+1)t-n(n+1)/2.$
We notice that $\hat{H}\in W_{\rm loc}^{1,\infty}(\R^N)$ and
$\hat{H}_p(p)p-\hat{H}=0$ a.e.
It follows that $H(x,p)=\hat{H}(p)+\ell (x)$
satisfies~\eqref{BS-ell} if $\ell$ is continuous but does not
satisfy~\eqref{BS} (even if  $\ell$ is Lipschitz continuous).
The Hamiltonian $H(x,p)=\underline{H}(x,p)+\hat{H}(p)+\ell (x)$
where $\underline{H}$ satisfies~\eqref{H3er} fulfills the assumptions
of Theorem~\ref{uni_grad} if the diffusion matrix 
satisfies~\eqref{cond-sig}-\eqref{sig-deg}
but such a case cannot be handled with the results of~\cite{bs01}. 
\end{ex}

\begin{rem}\label{extension-nonlin}
We can extend the results with easy adaptations to fully nonlinear equations like
\begin{eqnarray}\label{hypsigmatheta}
&& \e \phi^\e
+ \mathop{\rm sup}_{\theta\in\Theta}\{
  -{\rm trace}(A_\theta(x) D^2\phi^\e)
+\underline{H}_\theta(x, D\phi^\e )
+\overline{H}_\theta(x, D\phi^\e )
\}
=0, \quad x\in\T^N,
\end{eqnarray}
where $\Theta$ is a compact metric space and there exists a constant $C>0$
such that $A_\theta = \sigma_\theta\sigma_\theta^T$ satisfies 
\begin{eqnarray*}
|\sigma_\theta(x)|\leq C, \quad
|\sigma_\theta(x)-\sigma_\theta(y)|\leq C|x-y|,
\qquad x,y\in\T^N,\theta\in \Theta,
\end{eqnarray*}
and both 
$\underline{H}_\theta, \overline{H}_\theta$ are continuous satisfying,
for all $R>0,$ there exists a modulus
of continuity $m_R$ such that
\begin{eqnarray}\label{hypHtheta}
&&|{H}_\theta(x,0)|\leq C,\quad 
|{H}_\theta(x,p)-{H}_\theta(y,p)|\leq m_R(|x-y|),
\quad x,y\in\T^N, |p|\leq R,\theta\in \Theta.
\end{eqnarray}
Then

\begin{itemize}

\item Theorem~\ref{uni_grad} holds when $A_\theta$ satisfies~\eqref{sig-deg},
$\underline{H}_\theta$ satisfies~\eqref{H3er} and $\overline{H}_\theta$ satisfies~\eqref{BS-ell},
uniformly with respect to $\theta.$

\item Theorem~\ref{uni_grad_dege} holds in particular when $\underline{H}_\theta\equiv 0$
and $\overline{H}_\theta$ satisfies~\eqref{BS} uniformly with respect to $\theta.$

\end{itemize}
\end{rem}

\section{Asymmetric Systems}
\label{assym-sys}

We consider the weakly coupled system
\begin{eqnarray}\label{AEsys}
&& \left\{
\begin{array}{l}
e \phi_1^\e-{\rm trace}(A_1(x) D^2\phi_1)+H_1(x,D\phi_1 )+\phi_1-\phi_2=0,~~x\in\T^N,\\[2mm]
e \phi_2^\e-{\rm trace}(A_2(x) D^2\phi_2)+H_2(x,D\phi_2 )+\phi_2-\phi_1=0,
\end{array}
\right.
\end{eqnarray}
where the 
$A_{i}$'s and $H_i$'s satisfy the steady
assumption~\eqref{cond-sig}.

\subsection{Gradient bounds for systems}
\label{sec:grad-sys}

The aim is to obtain uniform gradient bounds (i.e., independent of $\e$) 
for~\eqref{AEsys} when 
$H_{i}=\underline{H}_{i}+\overline{H}_{i}$. We define
\begin{eqnarray}\label{defcalH}
\mathcal{H}:= \sup_{x\in\T^N,1\leq j\leq 2}|H_{j}(x,0)|.
\end{eqnarray}
and each equation satisfies a different set of assumptions.
\begin{eqnarray}\label{BS222sys}
&&\text{$A_{1}, \underline{H}_{1}, \overline{H}_{1}$ satisfy \eqref{sig-deg},\eqref{H3er},\eqref{BS-ell} respectively,}
\end{eqnarray}
\begin{eqnarray}\label{BS222sysbis}
&&\text{$\underline{H}_{2}\equiv 0$
and $\overline{H}_{2}$ satisfies \eqref{BS} 
with $|\overline{H}(\cdot,0)|$
replaced by $3\mathcal{H}.$}
\end{eqnarray}

Assumption~\eqref{BS222sys} means that the first equation is of uniformly elliptic type
with sublinear Hamiltonian whereas~\eqref{BS222sysbis} tells that the second one may be
degenerate with superlinear Hamiltonian. So the system is asymmetric. This case is
the one in interest in this work but let us mention that the case when both equations
satisfy either~\eqref{BS222sys} or~\eqref{BS222sysbis} is also possible with easier
arguments in the proof of the theorem which follows. See  Section~\ref{sec:exple-sys}
for examples and extensions.

\begin{thm}\label{uni_gradsy}
Assume~\eqref{BS222sys}-\eqref{BS222sysbis}. 
There exists a unique continuous viscosity solution
$\phi^\e=(\phi_1^\e,\phi_2^\e)\in C(\T^N)^2$ of \eqref{AEsys} 
and $K>0$ depending only on the $H_{i}$'s such that
\begin{eqnarray*}
|D\phi_i^\e|_\infty \le K \quad \text{for all } \e>0,~~ i=1,2.
\end{eqnarray*}
\end{thm}

Similarly to the case of scalar equations, the proof consists 
in two main steps: first, we prove an uniform bound for the 
oscillation and we then improve it to a uniform gradient bound.
The key lemmas are
\begin{lem}\label{uni_oscisy}
Under the hypotheses of Theorem~\ref{uni_gradsy}. 
Let 
$\phi^\e=(\phi_1^\e,\phi_2^\e)\in C(\T^N)^2$ be a solution of \eqref{AEsys}.
There exists $K>0$ depending only on the $H_{i}$'s such that
\begin{eqnarray*}
{\rm osc}(\phi_i^\e) \le K \quad \text{for all } \e>0,~~ i=1,2.
\end{eqnarray*}
\end{lem}
We skip the proof of Lemma~\ref{uni_oscisy} since it is similar to the case of scalar equations.

\begin{lem}\label{uniform3}
Under the hypotheses of Theorem~\ref{uni_gradsy}, let $\phi^\e$ be a 
continuous viscosity solution of~\eqref{AEsys}
and define $w^\e$ by
\begin{eqnarray*}
\exp(w^\e_i)=\phi^\e_i-\min_{\T^N} \phi^\e_i+1
\quad \text{for all $\e,$ $i$}.
\end{eqnarray*}
Then, there exists a constant $K$ independent of $\e$ 
such that $|Dw^\e_i|_\infty\le K$ for all $\e, i$.
\end{lem}

\begin{proof}[Proof of Theorem~\ref{uni_gradsy}]
It is sufficient to see that the proof of Theorem~\ref{uni_grad}
can be extended to systems like~\eqref{AEsys}.
For $i=1,2,$ we define $H_{in}$ like
in~\eqref{H-approche}.
By Perron's method for systems, see~\cite{ik91b}
and Theorem~\ref{comp-discont}, there exists a
continuous viscosity solution $\phi_n$ of~\eqref{AEsys}
where the $H_{i}$ are replaced with the $H_{in}$'s.
It is now possible to apply
Lemma~\ref{uniform3} to $\phi_n$ to obtain a gradient bound which
is independent of $\e, n.$ We conclude as in the proof of
Theorem~\ref{uni_grad}.
\end{proof}

\begin{proof}[Proof of Lemma \ref{uniform3}]
The proof is almost the same as the one for scalar equations. The main change 
is the estimate~\eqref{estimate}.  So we only write some important steps.

\noindent
1. Set $m_i=\min_{\T^N}\phi_i,$ the new equation satisfied by $w_i$ is

\begin{eqnarray*}
-{\rm trace}(A_{i}(x) D^2w_i)+\underline{G}_{i}(x, w_i,Dw_i)
+\overline{G}_i(x, w_i,Dw_i)
+b_i(w_i)+1-e^{w_{i+1}-w_i}=0,
\end{eqnarray*}
where 
\begin{eqnarray*}
&&\underline{G}_{i}(x,w,p)=e^{-w}\underline{H}_{i}(x,e^w p )-|\sigma_{i} (x)^Tp|^2,
\quad \overline{G}_{i}(x,w,p)=e^{-w}\overline{H}_{i}(x,e^w p ),\\
&&b_i(w)=e^{-w}[m_i-m_{i+1}+\e(m_i-1)]+\e,~~\text{ where we identify $m_3=m_1$ and $w_3=w_1$}.
\end{eqnarray*}
Consider
\begin{eqnarray}\label{maxsys123}
\max_{\T^N\times \T^N,\, i=1,2}\{ w_i(x)-w_i(y)-\Psi(|x-y|)\},
\end{eqnarray}
where $\Psi$ is of the form in~\eqref{def-phi7} with $A_1,A_2$ to be chosen later. 
We are done if the maximum is nonpositive.

\noindent
2. Otherwise, the maximum is positive and hence  is achieved at 
$(\overline{x},\overline{y})$ with
$\overline{x} \neq \overline{y}$ for some $i\in \{1,2\}.$
So we can write the viscosity inequalities
for the $i$th equation  at this point. 
For every $\varrho>0,$ there exist 
$(p,X) \in \overline{J}^{2,+}w_i(\overline{x}),$ and
$(p,Y) \in \overline{J}^{2,-}w_i(\overline{y})$
such that~\eqref{mat}-\eqref{mat-bis}-\eqref{mat-ter} hold
and~\eqref{ineq3}  
is replaced by
\begin{eqnarray}\label{ineq-visc-638}
&& \underline{\mathcal{G}}_{i}
+ \overline{\mathcal{G}}_{ i} 
+ \mathcal{B} + \mathcal{C}\leq 0,
\end{eqnarray}
where
\begin{eqnarray*}
&& \underline{\mathcal{G}}_{i}=
-{\rm tr}(A_{i}(\overline{x})X-A_{i}(\overline{y})Y)+\underline{G}_{i}(\overline{x},w_i(\overline{x}),p)-\underline{G}_{ i}(\overline{y},w_i(\overline{y}),p),\\
&& \overline{\mathcal{G}}_{ i}=
\overline{G}_{i}(\overline{x},w_i(\overline{x}),p)-\overline{G}_{i}(\overline{y},w_i(\overline{y}),p),\\
&& \mathcal{B}= b_i(w_i(\overline{x}))-b_i(w_i(\overline{y})),\\
&& \mathcal{C}=
e^{w_{i+1}(\overline{y})-w_i(\overline{y})}-e^{w_{i+1}(\overline{x})-w_i(\overline{x})},~~\text{ where we identify $w_3=w_1$}.
\end{eqnarray*}

Since the maximum in~\eqref{maxsys123} is positive, we have $\mathcal{C}\geq 0.$
Next, we claim that
\begin{eqnarray}\label{estimate}
m_i-m_{i+1}+\e m_i \le 3\mathcal{H},
\end{eqnarray}
see Step 5 for its proof. It follows easily that
\begin{eqnarray*}
\mathcal{B}\ge-3\mathcal{H}.
\end{eqnarray*}

Since we have different assumptions for each equation, 
we distinguish two cases: "uniformly elliptic" (Step 3) and "degenerate elliptic" (step 4) to get a contradiction in~\eqref{ineq-visc-638}.

\noindent
3. When $i=1$, since the first equation is "uniformly elliptic", we use the arguments of Step~4
in the proof of Lemma~\ref{uniform2}. We obtain
\begin{eqnarray}\label{abcf}
&& \underline{\mathcal{G}}_{1}+\overline{\mathcal{G}}_{1}
+ \mathcal{B} + \mathcal{C}\\\nonumber
&\geq&
-4\nu\Psi''-(\tilde{C}+2\underline{C}+\overline{C} 
+2|\sigma_1|_\infty |(\sigma_1)_x|_\infty |\overline{x}-\overline{y}|\Psi')\Psi'
-2\underline{C}-3\mathcal{H}.
\end{eqnarray}
With a suitable choice of $A_1,A_2$ in the definition of $\Psi,$
there exists $r\geq |\overline{x}-\overline{y}|$
such that $|\overline{x}-\overline{y}|\Psi'\leq \Psi(r)=\mathop{\rm max}_{j} {\rm osc}(\phi_j),$
which is bounded by Lemma~\ref{uni_oscisy}. It is then possible to make
the right-hand side of~\eqref{abcf} positive. 

\noindent
4.  When $i=2$, since the second equation is "degenerate elliptic" so we control the $0$th order terms
using the superlinear Hamiltonian. 
We repeat readily the  arguments of the proof of Lemma~\ref{uniform2}
in Section~\ref{sec:bound-scalar-dege} (case of degenerate elliptic equations)
estimating $\overline{\mathcal{G}}_{2}$ using now~\eqref{BS222sysbis}. We omit the details.

Finally
Steps 3 and 4 lead to a contradiction in~\eqref{ineq-visc-638}.

\noindent
5. It remains to prove~\eqref{estimate}. Let $x_1$ such that $\phi_1(x_1)=\min_{\T^N}\phi_1.$
Using the first equation of~\eqref{AEsys}, we get
\begin{eqnarray*}
\min_{\T^N}\phi_2-\min_{\T^N}\phi_1 \le \phi_2(x_1)-\min_{\T^N}\phi_1\le \e \phi_1(x_1)
+H_{1} (x_1, 0 ) 
\le 2\mathcal{H}.
\end{eqnarray*}
Taking into account~~\eqref{eps-phi-borne}, we get the desired inequality.
\end{proof}

\subsection{Strong maximum principle for systems}
\label{sec:smp}

The following extension of the strong maximum principle 
to parabolic systems 
is a crucial ingredient in the proof of the large time 
behavior.

\begin{thm} \label{smp-sys}
Assume that, for $i=1,2,$ 
$A_{i}(x)\geq \nu_i(x)I$ with $\nu_i\in C(\T^N),$ $\nu_i\geq 0$
and
\begin{eqnarray}\label{partition-tore}
&& \text{for all } x\in\T^N, \  \sum_{i=1,2}\nu_i(x)>0.
\end{eqnarray}
If  
$u$ is a continuous viscosity subsolution of
\begin{eqnarray}\label{Cbis}
&& \left\{
\begin{array}{ll}
\displaystyle\frac{\partial u_1}{\partial t}
-{\rm trace}(A_1(x)D^2 u_1)
-C|Du_1|+u_1-u_2=0, & x\in\T^N\!\times\! (0,+\infty), 
\\[3mm]
\displaystyle\frac{\partial u_2}{\partial t}
-{\rm trace}(A_2(x)D^2 u_2)
-C|Du_2|+u_2-u_1=0,
\\[3mm]
u_i(x,0)=u_{0i}(x), & x\in \T^N,
\end{array}
\right. 
\end{eqnarray}
which attains a maximum at 
$(\overline{x},\overline{t})\in \T^N\times (0,+\infty),$
then $u$ is constant.
\end{thm}

\begin{rem}\label{rem-smp} \ \\
(i) A new feature of the above result is that one only need a
partial nondegeneracy condition for the diffusion matrices $A_i$
in the following sense. At each point of $\T^N,$ there exists at
least one equation such that $A_i(x)$ is nondegenerate.
It can be interpreted using optimal control as follows. When 
considering the stochastic 
control problem associated with the equation in~\eqref{Cbis},
it means that the controlled process visits any open set of $\T^N$
almost surely for any open time interval (see Section~\ref{sec:control_provi} 
for the control interpretation).\\
(ii) This result contains, as a particular case, stationary systems.
\end{rem}

\begin{proof}[Proof of Theorem~\ref{smp-sys}] \ \\
\noindent 1. Suppose that
\begin{eqnarray}\label{max-total}
M:= \mathop{\rm sup}_{(x,t)\in\T^N\times [0,+\infty),\, j=1,2}
u_j(x,t) = u_i(\overline{x},\overline{t}), \quad \overline{t}>0.
\end{eqnarray}
We do a formal calculation. A rigorous one can be made using
classical viscosity techniques. At the point $(\overline{x},\overline{t}),$
we have $\frac{\partial u_i}{\partial t}=0,$ $Du_i=0$
and $-D^2 u_i\geq 0.$ From the $i$th equation, we obtain
\begin{eqnarray*}
u_i(\overline{x},\overline{t})-u_{i+1}(\overline{x},\overline{t}) \leq 0,~~\text{ where we identify $u_3=u_1$}.
\end{eqnarray*}
 Therefore $u_1(\overline{x},\overline{t})=u_2(\overline{x},\overline{t}).$
It follows that $(\overline{x},\overline{t})$ is the common maximum point of
$u_1, u_2$ in $\T^N\times [0,+\infty).$ 

\noindent 2. 
By~\eqref{partition-tore}, for all $x\in\T^N,$ there exists $j\in \{1,2\}$
such that $\nu_j(x)>0.$ By continuity, there exists $r_x>0$ such that
$\overline{B}(x,r_x)\subset \Omega_j:=\{x\in\T^N : \nu_j(x)>0\}.$
It follows
\begin{eqnarray*}
\T^N=\bigcup_{x\in\T^N} {B}(x,r_x)
\end{eqnarray*}
and, by compactness, there exists a finite covering
\begin{eqnarray}\label{new_partition}
\T^N=\bigcup_{p=1}^n B_p
\ \ \text{with, for all $p,$ $\overline{B}_p\subset \Omega_j$
for some $j\in\{1,2\}.$}
\end{eqnarray}
It follows that there exist $p,j$ such that 
$\overline{x}\in \overline{B}_p\subset \Omega_j.$
By continuity of $\nu_j$ and compactness, we have
\begin{eqnarray*}
&& \mathop{\rm inf}_{y\in \overline{B}_p, |\xi|=1}
\langle A_{j}(y)\xi,\xi \rangle
\geq \mathop{\rm inf}_{y\in \overline{B}_p} \nu_j(y)
=:\nu >0.
\end{eqnarray*}
Hence the $j$th equation is uniformly parabolic in $B_p.$
Moreover, the maximum in~\eqref{max-total} is also a maximum on
$B_p\times [0,+\infty).$
We set $u_k^M= u_k-M\leq 0$ for all $k.$ 
So we have
\begin{eqnarray*}
&& \frac{\partial u_{j}^M}{\partial t}
-{\rm trace}(A_j D^2 u_{j}^M)
-C|Du_{j}^M|+u_{j}^M\\
&\leq&
\frac{\partial u_{j}^M}{\partial t}
-{\rm trace}(A_{j} D^2 u_{j}^M)
-C|Du_{j}^M|+u_{j}^M-u_{j+1}^M \quad \text{identifying $u_3=u_1,$}\\
&=& \frac{\partial u_{j}}{\partial t}
-{\rm trace}(A_{j} D^2 u_{j})
-C|Du_{j}|+u_j-u_{j+1} \leq 0 \quad \text{in ${B}_p$.}
\end{eqnarray*}
By the strong maximum principle for viscosity solutions
of single parabolic equations
(see, e.g.,~\cite{dalio04}), we obtain $u_{j}^M\equiv 0$
in $B_p\times [0,+\infty).$
Coming back to Step 1, we infer that 
the maximum in~\eqref{max-total} is achieved for
$i=1,2$ at every $(y,t)\in B_p\times [0,+\infty).$ 

\noindent
3. By~\eqref{new_partition}, there exist $p'$ and $x'$ such that
$x'\in B_p\cap B_{p'}.$ It follows that the $u_j$'s
achieve their maximum over $B_{p'}\times [0,+\infty)$
at $(x',\overline{t}).$ Repeating Step 2, we conclude that
the $u_j$'s are constant in $B_{p'}\times [0,+\infty).$
From~\eqref{new_partition},
we conclude that $u_j\equiv M$ for all
$j$ and $(x,t)\in \T^N\times [0,+\infty).$
\end{proof}

\subsection{Ergodic problem}
The uniform gradient bound established for~\eqref{AEsys} allows us
to solve the ergodic problem.

The following assumption is used to ``linearize'' the system
in order to apply the strong maximum principle-Theorem~\ref{smp-sys}.
\begin{eqnarray}\label{hypA2bis}
&& \text{$H_{i}\in W_{\rm loc}^{1,\infty}(\T^N \times \R^N)$.}
\end{eqnarray}

\begin{thm}[Ergodic problem]\label{thm:ergod}
Suppose that the assumptions of Theorem~\ref{uni_gradsy} hold for~\eqref{AEsys}.
Then, there exists a solution $(c,v)\in \R^2\times W^{1,\infty}(\T^{N})^{2}$
of~\eqref{E}. 
The ergodic constant $c$ is unique and $c=(c_1 , c_1).$
If, in addition, the $H_{i}$'s 
satisfy~\eqref{hypA2bis}, then
$v$ is unique up to an additive constant vector.
\end{thm}

For $m=1,$ we find the classical results for scalar equations
(see \cite{lpv86, bs01}).

\begin{proof}[Proof of Theorem {\rm \ref{thm:ergod}}]
Let $\phi^\ep$ be the continuous solution of~\eqref{AEsys}
given by Theorem~\ref{uni_gradsy}. 
We first claim that there exists a constant $C'$ independent
of $\ep$ such that
\begin{eqnarray}\label{estimation-somme}
\left|\phi_i^\ep(x)-\phi_{i+1}^\ep(x)\right| \leq C', \quad x\in\T^N, i=1,2~\text{and we identify $\phi_3^\ep=\phi_1^\ep$.}
\end{eqnarray}
Indeed, choosing $\overline{x}_i$ as the maximum point of $\phi_i^\ep$, 
using the $i$th equation and~\eqref{eps-phi-borne},
we have
\begin{eqnarray*}
|\max\phi_1^\ep-\max\phi_2^\ep|
\leq
C.
\end{eqnarray*}
From Theorem \ref{uni_gradsy},
we obtain
\begin{eqnarray*}
&&\phi_i^\ep(x)-\phi_{i+1}^\ep(x)
= [\max\phi_i-\max\phi_{i+1}]
+[\phi_i^\ep(x)-\phi_i^\ep(\overline{x}_i)]+[\phi_{i+1}^\ep(\overline{x}_{i+1})-\phi_{i+1}^\ep(x)]\\
&\leq&
C+ {\rm osc}(\phi_1)+ {\rm osc}(\phi_2).
\end{eqnarray*}
The lower bound is established in the same way by introducing
$\underline{x}_i$ as the minimum point of~$\phi_i$.

Fix $x^*\in \T^N$ and
set $\rho_i^\ep := \phi_i^\ep(x^*)-\phi_{i+1}^\ep(x^*)$
and $v_i^\ep(x):= \phi_i^\ep(x)-\phi_i^\ep(x^*)$
for all $i=1,2.$
From Theorem \ref{uni_gradsy},~\eqref{eps-phi-borne} 
and~\eqref{estimation-somme}, using 
Ascoli-Arzela's theorem, there exists a sequence $\ep_k \to 0$ so that
\begin{eqnarray*}
&& v_{i}^{\ep_k}\to v_{i}, \qquad
\ep_k \phi_i^{\ep_k}\to -c_i,  \qquad
\rho_i^{\ep_k} \to \rho_i,
\end{eqnarray*}
uniformly on $\T^{N}$ 
as $k\to\infty$
for some $v_{i}\in W^{1,\infty}(\T^{N}),$
$c_i, \rho_i\in\R.$ Note that $c_i$ is a constant independent
of $x^*$ while $v_{i}, \rho_i$ depend on $x^*$.

Now, we rewrite~\eqref{AEsys} as
\begin{eqnarray}\label{nouveau-sys}
&& \left\{
\begin{array}{l}
\e\phi_1^\e-{\rm trace}(A_{1}(x) D^2v_1^\e)
+H_1(x,Dv_1^\e )
+v_1^\e-v_2^\e+\rho_1^\ep=0,\\[2mm]
\e\phi_2^\e-{\rm trace}(A_{2}(x) D^2v_2^\e)
+H_2(x,Dv_2^\e )
+v_2^\e-v_1^\e+\rho_2^\ep=0.
\end{array}
\right.
\end{eqnarray}
 Passing to the limit along the subsequence $\ep_k,$ by stability, we
obtain that $v=(v_1,v_2)\in W^{1,\infty}(\T^{N})^2$
is solution of
\begin{eqnarray*}
&& \left\{
\begin{array}{l}
-{\rm trace}(A_{1}(x) D^2v_1)+H_{1}(x,Dv_1 )
+v_1-v_2+\rho_1= c_1,\\[2mm]
-{\rm trace}(A_{2}(x) D^2v_2)+H_{2}(x,Dv_2 )
+v_2-v_1+\rho_2= c_2.
\end{array}
\right.
\end{eqnarray*}

Multiplying~\eqref{nouveau-sys} by $\ep$ and  passing to the limit along 
the subsequence $\ep_k,$ we get $c_1=c_2$. Note that $\rho_1^\ep+\rho_2^\ep=0$, so $\rho_1=-\rho_2$.
We conclude by setting $\tilde{v}_1(x)=v_1(x)+\rho_1,\tilde{v}_2=v_2(x)$.

We finally mention that we can easily see the uniqueness of the ergodic constant 
by the comparison principle for~\eqref{C} (Theorem~\ref{comp-cauchy-lip}).
We claim that the uniqueness up to constant of solutions comes
from the strong maximum principle (Theorem~\ref{smp-sys}).
Indeed, let $v, \tilde{v}$ be Lipschitz continuous solutions of~\eqref{E}
(the constant $c$ is the same by the above). Since $|Dv|_{\infty}, 
|D\tilde{v}|_{\infty}\leq K$ for some $K$ and 
$H_{i}\in W^{1,\infty}(\overline{B}(0,K))$, classical arguments
in viscosity solutions imply that $v-\tilde{v}$ is a viscosity subsolution
of the stationary version of~\eqref{Cbis} in $\T^N.$ Therefore
$v=\tilde{v}+C$ where $C\in \R^m$ is a constant.
\end{proof}

\subsection{Large time behavior result}


We first give a general result and then apply it to asymmetric systems.

\begin{thm}(Large time behavior)\label{LTBgen}
Suppose that~ \eqref{partition-tore}, \eqref{hypA2bis} hold. 
Suppose there exists a viscosity solution
$u\in C(\T^N\times [0,+\infty))^2$ of~\eqref{C}, a solution
$((c_1,c_1),v) \in \R^2 \times  W^{1,\infty}(\T^N)^2$ of the ergodic problem~\eqref{E}
and $K>0$ such that
\begin{eqnarray*}
|Du_i(\cdot,t)|_\infty,  \leq K,
\quad t\geq 0, i=1,2.
\end{eqnarray*}
Then, there exists a constant $\ell\in\R$ such that
$$
u(x,t)+(c_1,c_1)t+(\ell,\ell) \to v(x) \quad uniformly \ as \ t\to +\infty.
$$
\end{thm}

\begin{proof}[Proof of Theorem \ref{LTBgen}] 
Several parts of this proof are inspired by~\cite{bs01}.

\noindent 1.
Since $|Du_i(\cdot,t)|_\infty\leq K,$ we deduce 
from Theorem~\ref{comp-cauchy-lip} that $u$
is the unique continuous viscosity solution of~\eqref{C}. 

Defining $c=(c_1,c_1),$
we see that $v^\pm (x,t)=v(x)-ct\pm(C,C)$
is a super and subsolution, respectively, of~\eqref{C} 
for $C$ large enough.
Thanks to the comparison principle, we have
\begin{eqnarray*}
v_i(x)-C\leq u_i(x,t)+ct\leq v_i(x)+C, \quad x\in\T^N, i=1,2.
\end{eqnarray*}
It follows that $u_i(x,t)+ct$ is bounded; by the change of function
$u\to u+ct,$
we may assume without loss of generality that $u$ is bounded  and
the ergodic constant $c$ is 0. 

Set $m(t)=\max_{x\in\T^N,\, i=1,2}\{u_i(x,t)-v_i(x)\}.$ The comparison principle claims that $m$ is 
nonincreasing and hence, $m(t)\to \ell \text{ as } t \to \infty.$

\noindent 2.
A by-product of Step 1 
is that
$\{u(\cdot,t), t>0\}$ is relatively compact
in $W^{1,\infty}(\T^N)^2.$
So we can extract a sequence,
$t_j\to +\infty$ such that $u(\cdot,t_j)\to  \overline{u}\in W^{1,\infty}(\T^N)^2.$ 
Applying Theorem~\ref{comp-cauchy-lip} for~\eqref{C}, we obtain
\begin{eqnarray*}
&&|u_i(x,t+t_p)-u_i(x,t+t_q)|\leq 
\max_{y\in\T^N,\, j=1,2} |u_j(y,t_p)-u_j(y,t_q)|,
\quad x\in\T^N, \  t\geq 0, \ p,q\in\N,
\end{eqnarray*}
which proves that $(u(\cdot ,\cdot+t_p))_p$ is a Cauchy
sequence in $C(\T^N\times [0,+\infty))^2.$ We call $u_\infty$
its limit. Notice, on the one hand, that  
$|Du_{\infty i}(\cdot,t)|_\infty\leq K$ for all $i,t$
and, on the other hand, by stability, $u_\infty$ is solution of~\eqref{C}
with initial data $\overline{u}.$

\noindent 3. Using the uniform convergence of $u_i(.,t+t_j),$
we pass to the limit with respect to $j$ in 
\begin{eqnarray*}
m(t+t_j)=\max_{i,x}(u_i(x,t+t_j)-v_i(x))
\end{eqnarray*}
to obtain $\ell=\max_{i,x}(u_{\infty i}(x,t)-v_i(x))$ for any $t>0$.

\noindent 4. Since  $u_{\infty}$ is solution of~\eqref{C}
and $v$ is solution of~\eqref{E} and, up to increase $K,$ both are $K$-Lipschitz continuous
in $x,$ thanks to the Lipschitz continuity of $H_{i}$ with respect
to $p,$ see~\eqref{hypA2bis}, we obtain that $u_\infty -v$ is subsolution of~\eqref{Cbis}.
The strong maximum principle (Theorem~\ref{smp-sys}), then
implies 
\begin{eqnarray*}
\ell=u_{\infty i}(x,t)-v_{i}(x) \quad
(x,t) \in \T^N \times [0,+\infty), i=1,2.
\end{eqnarray*}
Noticing that $(\ell,\ell)+v(x)$ does not depend on the choice 
of subsequences, we obtain
$u_{i}(x,t)-\ell-v_{i}(x)\to 0$ uniformly in $x$ as $t \to \infty$, for $i=1,2$.
\end{proof}

We apply the previous result for the particular systems 
we studied in Section~\ref{assym-sys}.

\begin{cor}(Large time behavior)\label{LTB_provi}
Assume~\eqref{partition-tore}, \eqref{hypA2bis} and suppose that the assumptions
of Theorem~\ref{uni_gradsy} are in force. Then, for any initial condition
$u_0\in W^{1,\infty} (\T^N )^2,$

\begin{itemize}

\item[(i)]
The system~\eqref{C} has a unique viscosity 
solution $u\in C(\T^N\times [0,+\infty))^2$
and there exists $K>0$  such that
$|Du_i(\cdot,t)|_\infty\leq K$ for all $t\geq 0, i=1,2.$

\item[(ii)]
There exist a unique ergodic constant $c=(c_1,c_1)\in\R$
and $v\in W^{1,\infty}(\T^N)^2$ solution of~\eqref{E} 
such that $u(x,t)+ct \to v(x)$ uniformly as $t\to +\infty.$
\end{itemize}
\end{cor}

\begin{rem}
Remember that the classical assumptions on the Hamiltonians in~\eqref{C}
do not hold in order to have comparison and uniqueness of viscosity
solutions. It is why, we prove first the existence of a 
(Lipschitz) continuous in space viscosity solution of~\eqref{C}. To this
purpose, we need the initial condition to be Lipschitz continuous 
though this regularity assumption is not necessary to obtain the long time
behavior. 
\end{rem}

\begin{proof}[Proof of Corollary~\ref{LTB_provi}] \ \\
1. We introduce the truncated evolutive system~\eqref{C}
in $\T^N\times [0,+\infty)$ with the $H_{in}$ defined by~\eqref{H-approche-bis}.
From Theorem~\ref{comp-cauchy-lip} and Perron's method, there exists a
unique continuous viscosity solution
$u_n$ in $\T^N\times [0,+\infty).$

In the following, $M$ is a constant which may vary line to line but is independent
of $n.$\\
2. Applying Theorem~\ref{thm:ergod} with the $H_{in}$'s, we obtain the
existence of $(c_n,v_n)\in  \R^2\times W^{1,\infty}(\T^N)^2.$ 
Since the $H_{in}$'s satisfy~\eqref{BS222sys}-\eqref{BS222sysbis} with constants independent
of $n$ for large $n,$ and $c_n\to c$ when $n\to +\infty,$ we obtain 
from Theorem~\ref{uni_gradsy} that $|v_n|_\infty, |Dv_n|_\infty \leq M.$\\
3. Noticing that $v_n(x)-c_n t$ satisfies~\eqref{C} with the $H_{in}$'s, 
by Theorem~\ref{comp-cauchy-lip}, we obtain
\begin{eqnarray*}
v_{ni}(x) -M \leq  u_{ni}(x,t)+c_{ni} t\leq v_{ni}(x) +M,
\quad  (x,t)\in \T^N\times [0,+\infty), i=1,2.
\end{eqnarray*}
Therefore $|u_{n}(\cdot,t)+c_{n} t|_\infty, {\rm osc}(u_n(\cdot,t))\leq M.$
\\
4. We claim that $|Du_n (\cdot,t)|_\infty \leq M.$ To prove this fact,
we repeat the proof of Lemma~\ref{uniform3} for the evolutive
system~\eqref{C} with the $H_{in}$'s. Let $T>0$ 
and consider 
\begin{eqnarray*}
\max_{(x,y)\in(\T^N)^2, t\in [0,T], 1\le i \le m}\{ w_{ni}(x,t)-w_{ni}(y,t)-\psi(|x-y|)\},
\end{eqnarray*}
where
$e^{w_{ni}(x,t)}= u_{ni}(x,t)- \mathop{\rm min}_{\T^N} u_{ni}(\cdot,t) +1,$
and $\psi$ is given by~\eqref{def-phi7}.
If the maximum is nonpositive we are done. Otherwise 
it is positive and achieved at some $(\overline{x},\overline{y},\overline{t},i)$
with $\overline{x}\not=\overline{y}$ and we can
write the viscosity inequalities for the $i$th equation, see~\cite[Theorem 8.3]{cil92}:
For every $\varrho>0,$ there exist 
$(a,p,X) \in \overline{P}^{2,+}w_{ni}(\overline{x},\overline{t})$ and
$(b,p,Y) \in \overline{P}^{2,-}w_{ni}(\overline{y},\overline{t})$
such that~\eqref{visco-ineq185} holds since $a-b=0.$
We achieve a contradiction as
in the proof of Lemma~\ref{uniform3}.\\
\noindent
5. From Ascoli-Arzela's theorem and
the stability result for systems, by letting $n\to +\infty,$
we obtain a continuous viscosity $u$ solution of~\eqref{C}
in $\T^N\times [0,+\infty),$
with $|Du(\cdot,t)|_\infty\leq M.$ 
This solution is unique thanks to Theorem~\ref{comp-cauchy-lip}.
It ends the proof of (i).
The proof of (ii) is an immediate consequence of Theorems~\ref{thm:ergod} 
and~\ref{LTBgen}.
\end{proof}

\subsection{Examples and extensions}
\label{sec:exple-sys}

We give some examples such that our convergence result, 
Corollary~\ref{LTB_provi}, holds.

\begin{ex} (Systems with possibly degenerate equations and superlinear Hamiltonians)
Consider  
\begin{eqnarray*}
&& \left\{
\begin{array}{ll}
\displaystyle \frac{\partial u_1}{\partial t}- d_1(x)\Delta u_1 + a_1(x)|Du_1|^{1+\alpha_1}
+u_1 -u_2 =f_1(x), &
x\in\T^N\times (0,+\infty), 
\\[3mm]
\displaystyle \frac{\partial u_2}{\partial t}- d_2 (x)\Delta u_2 + a_2(x)|Du_2|^{1+\alpha_2} + u_2 -u_1 =f_2(x), &
x\in\T^N\times (0,+\infty), 
\\[3mm]
u_1(x,0)=u_{01}(x), \  u_2(x,0)=u_{02}(x), & x\in \T^N,
\end{array}
\right. 
\end{eqnarray*}
where $d_i\ge 0, a_i, f_i, u_{0i}$ are Lipschitz continuous, 
${\alpha_i}>0$ and $a_i>0.$ Suppose moreover that for any $x \in \T^N$, we have either $d_1(x)>0$ or $d_2(x)>0$ (this implies~\eqref{partition-tore}).   So,
the hypotheses of Corollary~\ref{LTB_provi} hold. 
\end{ex}

When there is at least one sublinear Hamiltonian, we need that the diffusion matrix of 
the corresponding equation is uniformly elliptic. But we 
permit other diffusion matrices to be degenerate everywhere.

\begin{ex} (Asymmetric systems with degenerate equations)
In the following system, the first equation is of sublinear type while the second is of superlinear type,
\begin{eqnarray}\label{sys-sous-sur124}
&& \left\{
\begin{array}{l}
\displaystyle \frac{\partial u_1}{\partial t}-{\rm trace}(A_1 D^2u_1) + \mathop{\rm sup}_{\theta\in\Theta}
\{ -\langle b_{\theta 1}(x) , Du_1\rangle -f_{\theta 1}(x)\} 
+u_1 -u_2 =f_1(x), 
\\[3mm]
\displaystyle \frac{\partial u_2}{\partial t}-{\rm trace}(A_2 D^2u_2)+a_2(x)|Du_2|^{1+\alpha} + u_2 -u_1 =f_2(x), \ \
x\in\T^N\times (0,+\infty), 
\\[3mm]
u_1(x,0)=u_{01}(x), \  u_2(x,0)=u_{02}(x), \ \ x\in \T^N,
\end{array}
\right. 
\end{eqnarray}
where the functions $b_{\theta 1}, f_{\theta 1}, a_2, f_2, u_{01}, u_{02}$ satisfy
\begin{eqnarray}\label{regb123}
&& |f(x)|\leq C, \quad
|f(x)-f(y)|\leq C|x-y|,\qquad x,y\in\T^N,\theta\in\Theta
\end{eqnarray}
and $\alpha>0$ and $a_2(x)>0.$
If we suppose moreover that $A_1>0$ in $\T^N$ (notice that we only assume $A_2 \ge 0$), 
then the hypotheses of Corollary~\ref{LTB_provi} hold. 
\end{ex}

We end with some possible extensions to fully nonlinear systems with $m\geq 2$
equations of the form~\eqref{sys-general} with 
$H_{\theta i}=\underline{H}_{\theta i}+\overline{H}_{\theta i}.$ We assume that
the system is monotone and irreducible, i.e., the coupling matrix
$D=(d_{ij})_{1\leq i,j\leq m}$ satisfies
\begin{eqnarray*}
&&  d_{ii}\ge 0, \quad d_{ij}\leq 0 \ {\rm for} \ i\not= j\quad 
{\rm and} \quad \sum_{j=1}^{m}d_{ij}= 0 \ {\rm for \ all  \ } i,\\
&&\text{for all subset  $\mathcal{I}\varsubsetneq \{1,\cdots ,m\}$, 
there exists $i\in \mathcal{I}$ and  $j\not\in  \mathcal{I}$ such that $d_{ij}\not= 0$.} \nonumber
\end{eqnarray*}
In this case, it is possible to find $\Lambda\in\R^m$
with positive components 
$\Lambda_i >0$ such that $D^T\Lambda=0$ (\cite{clln12}) and to define
$\lambda_D:= \mathop{\rm max}_{1\leq i\leq m}\frac{1}{\Lambda_i}
\sum_{j\not= i}\Lambda_j.$
Then it is possible to generalize the previous results:
We assume that $A_{\theta i}$ satisfies~\eqref{hypsigmatheta}
and $\underline{H}_{\theta i} ,\overline{H}_{\theta i}$
satisfy~\eqref{hypHtheta}.

\begin{itemize}

\item Theorem~\ref{uni_gradsy} holds if, for each $1\leq i\leq m,$ the $i$th-equation
satisfies, uniformly with respect to $\theta,$ 
either~\eqref{BS222sys} or~\eqref{BS222sysbis} with $|\overline{H}(\cdot,0)|$
replaced by $(2\lambda_D +1)\mathcal{H}$ where $\mathcal{H}$ is defined in~\eqref{defcalH}
with a supremum over all $1\leq j\leq m,$ $\theta\in\Theta.$ Notice that, when $m=1$ (resp. $m=2$),
$\lambda_D=0$ (resp. $\lambda_D=1$) and
we recover exactly Theorem~\ref{uni_grad_dege} (resp. Theorem~\ref{uni_gradsy}).

\item Theorem~\ref{thm:ergod} holds under the assumptions above
and~\eqref{hypA2bis} uniformly with respect to~$\theta.$

\item Corollary~\ref{LTB_provi} holds under the assumptions above
and~\eqref{partition-tore} with $A_{\theta i}\geq \nu_i I.$

\end{itemize}

\section{Appendix}

\subsection{Proof of Lemma~\ref{diff-tracet}}

From~\eqref{mat}, for every $\zeta, \xi\in\R^N,$ we have
\begin{eqnarray*}
\langle X\zeta,\zeta \rangle - \langle Y\xi,\xi\rangle
\leq \Psi'\langle \zeta-\xi,B(\zeta-\xi) \rangle 
+ \Psi'' \langle \zeta-\xi,(q\otimes q)(\zeta -\xi)\rangle
+O(\varrho).
\end{eqnarray*}
We estimate ${\rm trace}(A(\overline{x})X)$
and ${\rm trace}(A(\overline{y})Y)$ using
two orthonormal bases $(e_1,\cdots ,e_N)$ and
$(\tilde{e}_1,\cdots ,\tilde{e}_N)$ in the following way:
\begin{eqnarray}\label{estimT}
T:= {\rm trace}(A(\overline{x})X
-A(\overline{y})Y)
&=&
\sum_{i=1}^N \langle X\sigma(\overline{x}) e_i, \sigma(\overline{x}) e_i \rangle
- \langle Y\sigma(\overline{y}) \tilde{e}_i, 
\sigma(\overline{y}) \tilde{e}_i \rangle
\nonumber\\
&\leq &
\sum_{i=1}^N \Psi'\langle \zeta_i,B\zeta_i\rangle
+ \Psi''\langle \zeta_i,(q\otimes q)\zeta_i\rangle+O(\varrho)
\nonumber\\
&\leq &
\Psi''\langle \zeta_1,(q\otimes q)\zeta_1\rangle
+ \sum_{i=1}^N \Psi'\langle \zeta_i,B\zeta_i\rangle+O(\varrho),
\end{eqnarray}
where we set $\zeta_i= \sigma(\overline{x}) e_i-\sigma(\overline{y}) \tilde{e}_i$
and noticing that 
$\Psi''\langle \zeta_i,(q\otimes q)\zeta_i\rangle =\Psi'' \langle \zeta_i, q\rangle^2
\leq 0$ since $\Psi$ is concave.

We now build suitable bases in two following cases.
In the case where $\sigma$ is degenerate, we choose any orthonormal basis such that $e_i= \tilde{e}_i.$ It follows
\begin{eqnarray*}
T&\leq & 
\sum_{i=1}^N \Psi'\langle (\sigma(\overline{x})-\sigma(\overline{y}))e_i 
,B(\sigma(\overline{x})-\sigma(\overline{y}))e_i\rangle+O(\varrho)\\
&\leq &
\Psi' N |\sigma(\overline{x})-\sigma(\overline{y})|^2|B|+O(\varrho)\\
&\leq & \Psi' N |\sigma_x|_\infty^2 |\overline{x}-\overline{y}|+O(\varrho)
\end{eqnarray*}
from~\eqref{cond-sig} and since $|B|\leq 1/|\overline{x}-\overline{y}|.$

When~\eqref{sig-deg} holds, i.e.,
$A(x)\geq \nu I$ for every $x,$ we can set
\begin{eqnarray*}
e_1= \frac{\sigma(\overline{x})^{-1}q}{|\sigma(\overline{x})^{-1}q|},
\quad \tilde{e}_1= -\frac{\sigma(\overline{y})^{-1}q}{|\sigma(\overline{y})^{-1}q|},
\quad \text{where $q$ is given by~\eqref{mat-bis}}.
\end{eqnarray*}
If $e_1$ and $\tilde{e}_1$ are collinear, 
then we complete the basis with orthogonal unit vectors 
$e_i=\tilde{e}_i\in e_1^\perp,$ $2\leq i\leq N.$ Otherwise, in the plane
${\rm span}\{e_1, \tilde{e}_1\},$ we consider a rotation $\mathcal{R}$
of angle $\frac{\pi}{2}$ and define
\begin{eqnarray*}
e_2=\mathcal{R}e_1, \quad 
\tilde{e}_2 =-\mathcal{R}\tilde{e}_1.
\end{eqnarray*}
Finally, noticing that ${\rm span}\{e_1, e_2\}^\perp
={\rm span}\{\tilde{e}_1, \tilde{e}_2\}^\perp,$ we can complete the orthonormal
basis with unit vectors $e_i=\tilde{e}_i\in {\rm span}\{e_1, e_2\}^\perp,$ $3\leq i\leq N.$

From~\eqref{sig-deg}, we have
\begin{eqnarray}\label{ineg-sig}
\nu \leq \frac{1}{|\sigma(x)^{-1}q|^2}\leq |\sigma|_\infty^2\le
C^2.
\end{eqnarray}
 It follows
\begin{eqnarray*}
\langle \zeta_1,(q\otimes q)\zeta_1\rangle =
\left( \frac{1}{|\sigma(\overline{x})^{-1}q|}
+\frac{1}{|\sigma(\overline{y})^{-1}q|} \right)^2\geq 4\nu.
\end{eqnarray*}
From~\eqref{mat-bis}, we deduce $Bq=0.$ 
Therefore
\begin{eqnarray*}
\langle \zeta_1,B\zeta_1\rangle =0.
\end{eqnarray*}
For $3\leq i\leq N,$ using~\eqref{cond-sig},
\begin{eqnarray*}
\langle \zeta_i,B\zeta_i\rangle
= \langle (\sigma(\overline{x})-\sigma(\overline{y}))e_i,
B(\sigma(\overline{x})-\sigma(\overline{y}))e_i\rangle
\leq |\sigma_x|_\infty^2 |\overline{x}-\overline{y}|\leq \sqrt{N} C^2.
\end{eqnarray*}
since $|B|\leq 1/|\overline{x}-\overline{y}|$ and $|\overline{x}-\overline{y}|\leq \sqrt{N}.$
We have
\begin{eqnarray*}
|\zeta_2|
= |(\sigma(\overline{x})-\sigma(\overline{y}))\mathcal{R}e_1 +
\sigma(\overline{y})\mathcal{R}(e_1+\tilde{e}_1)|
\leq C|\overline{x}-\overline{y}| + C |e_1+\tilde{e}_1|.
\end{eqnarray*}
It remains to estimate 
\begin{eqnarray*}
|e_1+\tilde{e}_1|
&\leq& \frac{1}{|\sigma_\theta(\overline{x})^{-1}q|}
|\sigma(\overline{x})^{-1}q-\sigma(\overline{y})^{-1}q|
+ |\sigma\overline{y})^{-1}q|
\left|\frac{1}{|\sigma(\overline{x})^{-1}q|}-\frac{1}{|\sigma(\overline{y})^{-1}q|}\right|\\
&\leq &
\frac{2|\sigma|_\infty |\sigma_x|_\infty^2}{\nu}  |\overline{x}-\overline{y}|
= \frac{2C^3}{\nu} |\overline{x}-\overline{y}|,
\end{eqnarray*}
from~\eqref{ineg-sig} and $|(\sigma^{-1})_x|_\infty\leq  |\sigma_x|_\infty^2/\nu.$

From~\eqref{estimT}, we finally obtain the conclusion
$T\leq 4\nu\Psi'' +\tilde{C}\Psi'+O(\varrho)$
where
\begin{eqnarray}\label{def-ctilde}
&& \tilde{C}=\tilde{C}(N,\nu,|\sigma|_\infty, |\sigma_x|_\infty)
:= C^2\sqrt{N}(N-2 +(1+\frac{2C^3}{\nu})^2)
\text{ with $C:=\max\{|\sigma|_\infty,|\sigma_x|_\infty\}$.}
\end{eqnarray}

\subsection{Comparison principles
for stationary systems}

The following results are stated in~\cite{bs01} in the case
of a scalar equation without control. We state them for systems.

\begin{thm} \label{comp-avec-lip}
Let $\psi^\e\in USC(\T^N)^2$ and $\phi^\e\in LSC(\T^N)^2$
be respectively a viscosity subsolution and supersolution
of~\eqref{syseps22intro}.
Assume that either $\psi^\e$ 
or $\phi^\e\in W^{1,\infty}(\T^N)^2.$
Then $\psi_i \leq \phi_i$ in $\T^N$ for $i=1,2.$
\end{thm}

A useful consequence of Theorem~\ref{comp-avec-lip} is
that, for any viscosity solution $\phi^\e$ of~\eqref{syseps22intro},
\begin{eqnarray} 
&& \label{eps-phi-borne}
|\e \phi_i^\e |_\infty\leq \mathop{\rm sup}_{j=1,2} |H_j(\cdot,0)|_\infty=:\mathcal{H},
\quad i=1,2.
\end{eqnarray}

\begin{proof}[Proof of Theorem~\ref{comp-avec-lip}]
We only give a sketch of proof since it is classical.
We suppose without loss of generality that
$\mathop{\rm sup}_i |D\psi_i|_\infty =:L<\infty.$
Set
\begin{eqnarray*}
&& M_0=\mathop{\rm max}_{x\in\T^N,\, j=1,2} \{\psi_j(x)-\phi_j(x)\},\\ 
&& M_\alpha=\mathop{\rm max}_{x,y\in\T^N,\, j=1,2} \{\psi_j(x)-\phi_j(y)
-\alpha^2|x-y|^2\}
=\psi_i(\overline{x})-\phi_j(\overline{y})
-\alpha^2|\overline{x}-\overline{y}|^2.
\end{eqnarray*}
We have some classical estimates
\begin{eqnarray*}
\alpha^2|\overline{x}-\overline{y}|^2\mathop{\to}_{\alpha\to \infty}0,
\quad
|p|\leq L \text{ where } p=2\alpha^2(\overline{x}-\overline{y}),
\quad
\mathop{\rm lim\,sup}_{\alpha\to \infty} M_\alpha =M_0.
\end{eqnarray*}
After subtracting  write the viscosity inequalities at $(\overline{x},\overline{y})$ of ~\eqref{syseps22intro}, we get
\begin{eqnarray*}
&& \e (\psi_i(\overline{x})-\phi_i(\overline{y}))
-{\rm trace}(A_{i}(\overline{x})X
- A_{i}(\overline{y})Y)
+H_{i}(\overline{x},p)-H_{i}(\overline{y},p)\}\\
&& \hspace*{8cm}+[\psi_i(\overline{x})-\phi_i(\overline{y})]-[\psi_{i+1}(\overline{x})-\phi_{i+1}(\overline{y})]
\leq 0.
\end{eqnarray*}
Since $\psi_j(\overline{x})-\phi_j(\overline{y})\leq \psi_i(\overline{x})-\phi_i(\overline{y}),$
we get $[\psi_i(\overline{x})-\phi_i(\overline{y})]-[\psi_{i+1}(\overline{x})-\phi_{i+1}(\overline{y})]\geq 0.$

\noindent Moreover $-{\rm trace}(A_{i}(\overline{x})X
- A_{i}(\overline{y})Y)
\geq o_\alpha(1)$.
Therefore $\e (\psi_i(\overline{x})-\phi_i(\overline{y}))
\leq o_\alpha(1),
$
this yields $M_0\leq 0$ as desired.
\end{proof}

In the proofs of Theorems~\ref{uni_grad} and~\ref{uni_gradsy}, we need
a discontinuous comparison result for truncated system when
$H_{i}$ in~\eqref{syseps22intro} is replaced by 
\begin{eqnarray}\label{H-approche-bis}
H_{in}:= \left\{ 
\begin{array}{ll}
H_{i}(x,p) & \text{if } |p|\leq n,\\
H_{i} (x,n\frac{p}{|p|}) & \text{if } |p|\geq n.
\end{array}
\right. 
\end{eqnarray}
Due to the truncation in the gradient variable, we do not need
to assume that the sub- or supersolution is Lipschitz continuous.
The proof follows the same lines as above.
We obtain

\begin{thm} \label{comp-discont}
Let $\psi\in USC(\T^N)^2$ and $\phi\in LSC(\T^N)^2$
be respectively a viscosity subsolution and supersolution
of~\eqref{syseps22intro} with $H_{in}$ defined by~\eqref{H-approche-bis}.
Then $\psi_i \leq \phi_i$ in $\T^N$ for all $i.$
\end{thm}

\subsection{Comparison for the evolutive problem}

\begin{thm} \label{comp-cauchy-lip}
Let $u\in USC(\T^N\times [0,+\infty))^2$ and $v\in LSC(\T^N\times [0,+\infty))^2$
be respectively a viscosity subsolution and supersolution
of~\eqref{C}.
Assume that either $u(\cdot ,t)$ or
$v(\cdot ,t)\in W^{1,\infty}(\T^N)^2$ uniformly in $t\in [0,+\infty).$
Then
\begin{eqnarray}\label{form-comp-evol}
&& u_i(x,t)-v_i(x,t)\leq \mathop{\rm sup}_{y\in\T^N, j=1,2} (u_j(y,0)-v_j(y,0))^+,
\ (x,t)\in \T^N\times [0,+\infty), \quad i=1,2.
\end{eqnarray}
\end{thm}

\begin{thm} \label{comp-discont-evol}
Let $u\in USC(\T^N\times [0,+\infty))^2$ and $v\in LSC(\T^N\times [0,+\infty))^2$
be respectively a viscosity subsolution and supersolution
of~\eqref{C} with $H_{in}$ defined by~\eqref{H-approche-bis}.
Then \eqref{form-comp-evol} holds.
\end{thm}

The proofs are easy adaptations of the proofs of Theorems~\ref{comp-avec-lip}
and~\ref{comp-discont} in the case of degenerate parabolic systems.

\subsection{A Control-theoretic interpretation for~\eqref{sys-sous-sur124}}
\label{sec:control_provi}

We can give an interpretation of weakly coupled systems as 
dynamic programming equations of hybrid systems with pathwise
stochastic trajectories with random switching, see
\cite{fz98}.

Let  $ (\Omega, \mathcal{F}, (\mathcal{F}_t)_{t\geq 0}, {\P})$ be a
filtered probability space, $W_t$ be a $\mathcal{F}_t$-adapted
standard $N$-Brownian motion such that $W_0=0$ a.s.
Consider the controlled random evolution process
$(X_t,\nu_t)$ with dynamics
\begin{equation*}
\left\{
\begin{array}{l}
dX_t =  b_{\theta_t \nu_t}(X_t)dt+\sqrt{2}\, \sigma_{\nu_t}(X_t)dW_t, \ \ t> 0,\\
(X_0,\nu_0)=(x,i) \in \T^N\times \{1,2\},
\end{array}
\right.
\end{equation*}
where the control law  $\theta_t :[0,\infty)\to \Theta$ is a measurable function.

For every $\theta_t$
there exists a $\mathcal{F}_t$-adapted solution $(X_t,\nu_t),$ 
where $X_t:[0,\infty)\to \T^N$
is piecewise $C^1$ and $\nu_t$  is a continuous-time Markov chain
with state space $\{1,2\}$ and probability transitions given by
\begin{eqnarray}\label{trans-prob}
&& 
\begin{array}{l}
\P\{\nu_{t+\Delta t}=2\,|\, \nu_t=1\}=\Delta t+o(\Delta t),\\
\P\{\nu_{t+\Delta t}=1\,|\, \nu_t=2\}=\Delta t+o(\Delta t).
\end{array}
\end{eqnarray}

We introduce the value functions of the optimal control problems
\begin{equation*}
    u_i(x,t)=\inf_{\theta_t\in L^\infty([0,t],\Theta)}\E_{x,i}
\{\int_0^t f_{\theta_s \nu_s}(X_s)ds
+u_{0 \nu_t} (X_t)\},
\quad i=1,2,
\end{equation*}
where $\E_{x,i}$ denote the expectation of a trajectory starting at
$x$ in the mode $i.$

The function $u=(u_1,u_2)$ satisfies the system
\begin{eqnarray*}
\left\{
  \begin{array}{l}
\displaystyle
\frac{\partial u_1}{\partial t}
-{\rm trace}(\sigma_{1}(x)\sigma_{1}(x)^TD^2 u_1)
+ \mathop{\rm sup}_{\theta\in \Theta}\{
-\langle b_{\theta 1}(x), Du_1\rangle -f_{\theta 1}(x)\}
+ u_1 -u_2= 0 \\[5pt]
\displaystyle
\frac{\partial u_2}{\partial t}
-{\rm trace}(\sigma_{2}(x)\sigma_{2}(x)^TD^2 u_2)
+ \mathop{\rm sup}_{\theta\in \Theta}\{
-\langle b_{\theta 2}(x), Du_2\rangle -f_{\theta 2}(x)\}
+ u_2 -u_1= 0 \\[5pt]
u_{1}(x,0)=u_{0 1}(x), \ u_{2}(x,0)=u_{0 2}(x).
  \end{array}
\right.
\end{eqnarray*}
By choosing $\Theta=\R^N,$
setting $A_i=\sigma_i\sigma_i^T,$ $i=1,2$ and assuming that 
$b_{\theta 1}, f_{\theta 1}, u_{0 1}, u_{0 2}$ satisfy~\eqref{regb123},
$b_{\theta 2}=2\theta$ and $f_{\theta 2}=|\theta|^2-f_2$
with $f_2$ continuous,
we obtain~\eqref{sys-sous-sur124} with $a_2=1$ and $\alpha=1.$
If moreover $A_1>0,$ then 
the first equation is uniformly parabolic of sublinear 
type and the second one is
possibly degenerate superlinear with a quadratic Hamiltonian. 
The assumptions of Theorems~\ref{uni_gradsy},~\ref{thm:ergod}
and~\ref{LTB_provi} hold. 

Roughly speaking, the Lipschitz regularization of $u_1$ is provided by the nondegenerate
diffusion in mode 1 whatever the bounded drift does. While the Lipschitz regularity of
$u_2$ comes from the controlable unbounded drift in mode 2 even if the diffusion degenerates.
Assumption~\eqref{partition-tore} is obviously satisfied
since $\nu_1(x):= {\rm min}_{|\xi|=1} \langle A_1(x)\xi,\xi\rangle >0$ for all $x\in\T^N.$
So the strong maximum principle holds and we have the
convergence when $t\to +\infty.$ The point is that the nondegeneracy of $A_1$ implies
also the convergence for $u_2.$
This seems to be due to the combined effects of the nondegenerate Brownian motion in
mode 1 together with the number of switchings which tends to $+\infty$ as
$t\to +\infty$ (since the transition probabilities~\eqref{trans-prob} are positive
and independent of $t$). It follows that the process visits any open subset of
$\T^N$ in each mode as $t\to +\infty$ yielding the convergence to an equilibrium
state for $(u_1,u_2).$  


\vspace*{-0.1cm}

\end{document}